\documentclass[11pt]{amsart}
\usepackage{color}

\usepackage{array}

\usepackage{amsmath, amsfonts, amssymb, amsthm, faktor}

\usepackage[all]{xy}
\usepackage{tikz}
\usetikzlibrary{math}
\usetikzlibrary{patterns}
\usepackage{caption}
\usepackage{subfig}
\usepackage{cite}
\usepackage{caption}
\usepackage{multirow}
\usetikzlibrary{arrows,chains,matrix,positioning,scopes,decorations.pathreplacing,
decorations.pathmorphing,decorations.markings,arrows.meta}

\usepackage{aliascnt}
\usepackage[colorlinks, linktocpage, allcolors=black,breaklinks]{hyperref}

\usepackage{comment}
\usepackage{enumitem}

\newtheorem{theorem}{Theorem}[section]
\newtheorem{lemma}[theorem]{Lemma}

\newtheorem{proposition}[theorem]{Proposition}
\newtheorem{corollary}[theorem]{Corollary}

\theoremstyle{definition}
\newtheorem{definition}[theorem]{Definition}

\usepackage{color}
\newenvironment{revvy}{\color{magenta}}{}
\newenvironment{revdw}{\color{blue}}{}
\newenvironment{revsg}{\color{red}}{}

\def\by{\begin{revvy}}
\def\ey{\end{revvy}}
\def\bd{\begin{revdw}}
\def\ed{\end{revdw}}
\def\bg{\begin{revsg}}
\def\eg{\end{revsg}}

\def\Z{\mathbb{Z}}

\def\Z{\mathbb{Z}}

\def\a{m}
\def\bsft{\mathfrak{b}}

\def\e#1{e}

\begin{document}
\title{Borel combinatorics of Schreier graphs of $\Z$-actions}

\author{Su Gao}
\address{School of Mathematical Sciences and LPMC, Nankai University, Tianjin 300071, P.R. China}
\email{sgao@nankai.edu.cn}
\thanks{The authors acknowledge the partial support of their research by the National Natural Science Foundation of China (NSFC) grants 12271263 and 12250710128, and by the Fundamental Research Funds for the Central Universities.}

\author{Yingying Jiang}
\address{School of Mathematical Sciences and LPMC, Nankai University, Tianjin 300071, P.R. China}
\email{yyjiangmath@mail.nankai.edu.cn}

\author{Tianhao Wang}
\address{School of Mathematics and Computer Science, Quanzhou Normal University, Quanzhou 362000, P.R. China}
\email{tianhao\_wang@qq.com}

\begin{abstract} In this paper we consider the Borel combinatorics of Schreier graphs of $\Z$-actions with arbitrary finite generating sets. We focus mostly on the problem of computing the Borel chromatic numbers of the Schreier graphs on the free part of the shift space $2^\mathbb{Z}$ with a finite symmetric generating set $S$ of $\mathbb{Z}$. Our main results include some general lower and upper bounds, as well as the exact values in some special cases. These values are sometimes the same as the chromatic numbers of the distance graphs on $\mathbb{Z}$ given by $S$, and more often different from them.
\end{abstract}



\maketitle

\section{Introduction}
Borel combinatorics studies combinatorial properties of topological graphs or more generally Borel graphs. Since the seminal paper of Kechris, Solecki and Todorcevic \cite{KST} on Borel chromatic numbers, many researchers have contributed to this new and rapidly growing area of research; a comprehensive survey can be found in the recent book by Kechris and Marks \cite{KM}.

One of the reasons Borel combinatorics is interesting to researchers is that one obtains different answers to combinatorial questions when imposing definability requirements on the combinatorial objects. For instance, Laczkovich realized early on that the graph on the unit circle defined by a rotation by a fixed irrational angle has chromatic number $2$ (since every connected component is isomorphic to a bi-infinite path) but its Borel chromatic number is necessarily $3$, that is, any set of colors one needs to use to find a Borel proper coloring function from this graph must have at least $3$ elements. 

More recently, it was discovered by Gao, Jackson, Krohne and Seward \cite{GJKS, GJKS2} that there is a further difference between Borel combinatorics and continuous combinatorics for topological graphs for nontrivial reasons. For instance, the above graph considered by Laszkovich does not have a sensible continuous chromatic number because the space is connected. But if one considers the space with one connected component removed, then it becomes zero-dimensional and the continuous chromatic number makes sense and can be seen to be $3$ (while the Borel chromatic number remains $3$). Answering a question of \cite{KST}, the authors of \cite{GJKS2} considered the free part of the $\Z^2$-action on the shift space $2^{\Z^2}$, and showed that the Borel chromatic number is $3$; in contrast, they also showed  (see \cite{GJKS}) that the continuous chromatic number is $4$. A similar difference exists for edge chromatic numbers of this graph (see \cite{BHT, CU, GJKS, GWW, GR, Weil}).

The computations of chromatic numbers and edge chromatic numbers belong to a class of questions known as {\em locally checkable labelling} (\emph{LCL}) problems (see, e.g., \cite{GR, GR2}), a concept that had been studied in distributed computing in computer science (see, e.g. \cite{BHKetc}). Bernshteyn's influential work \cite{Ber} connected the distributed computing and Borel combinatorics together. There are now many results that characterize the descriptive-set-theoretic complexity of LCL problems by the computational complexity of local algorithms that solve the problem in a distributed fashion. 

In this paper we study the Borel chromatic numbers of Schreier graphs that come from $\Z$-actions but with arbitrary finite sets of generators. In fact we focus on the Borel chromatic numbers of the free part of the $\mathbb{Z}$-action on the shift space $2^\mathbb{Z}$, which is denoted as $F(2^\mathbb{Z})$, for some generating sets. In general, if $S$ is a finite generating set of $\Z$ that is symmetric (i.e., $S=-S$) and $0\not\in S$, we obtain a Schreier graph $G(F(2^\mathbb{Z}), S)$ on $F(2^\Z)$ by putting an edge between $x$ and $y$ if and only if there is $s\in S$ such that $s\cdot x=y$. 

By considering general LCL problems in this context, Greb\'{\i}k and Rozho\v{n} \cite[Section 2]{GR} have shown a fundamental result which implies that the Borel chromatic number $\chi_B(S)$ of $G(F(2^\mathbb{Z}), S)$ is the same as the continuous chromatic number of $G(F(2^\mathbb{Z}), S)$, and in addition, the function $\chi_B(S)$ is computable with $S$ as input ($S$ can be coded as a natural number). Even if $\chi_B(S)$ is a computable function, it is still of interest to know the exact values of this function for concrete generating sets. 

If $S$ is a finite symmetric generating set of $\mathbb{Z}$, we denote the Cayley graph of the marked group $(\mathbb{Z}, S)$ by $G(\mathbb{Z}, S)$. Namely, there is an edge between $m$ and $n$ in $\mathbb{Z}$ if and only if $m-n\in S$. This graph is known as the {\em distance graph} given by $S$, and has been extensively studied in graph theory and combinatorics. Notably,  the distance graph was first defined and studied by Eggleton, Erd\H{o}s, and Skilton \cite{EES} and subsequently studied by many authors \cite{EES90, Wal, CCH97,DZ97, KK98,  Zhu98,  CLZ,LZ99,Voi99,  HC00, Liu00,Zhu02,WL04, LL05, BS08,   OS24, CK25}.  We let $\chi(G(\mathbb{Z}, S))$ denote the chromatic number of $G(\mathbb{Z}, S)$. It is interesting to compare some of our results on 
the Borel chromatic numbers $\chi_B(S)$ and the classical results on the chromatic numbers $\chi(G(\mathbb{Z}, S))$. 

In this paper we obtain some upper and lower bounds for $\chi_B(S)$, and in some special cases compute the exact values of $\chi_B(S)$. The following are some of the main results we obtain.

\begin{theorem}\label{thm:1.2} Let $S=\{\pm a_1,\dots, \pm a_n\}$ with $\gcd(a_1,\dots, a_n)=1$. Let $\chi_B(S)$ be the Borel chromatic number of $G(F(2^\mathbb{Z}), S)$. Then the following hold.
\begin{enumerate}
\item[\rm (1)] $3\leq \chi_B(S)\leq \left\lfloor\displaystyle\frac{3}{2}n\right\rfloor +1$.
\item[\rm (2)] If $S\cap (k+1)\Z=\varnothing$ for some $k\geq 1$, then $\chi_B(S)\leq k+2$.
\item[\rm (3)] If $(a_1,\dots, a_n)=(1,\dots, n)$, then $\chi_B(S)=n+2$.
\item[\rm (4)] If $n=2$ and $(a_1, a_2)\neq (1,2)$, then $\chi_B(S)=3$.
\end{enumerate}
\end{theorem}

The computation of $\chi(G(\mathbb{Z}, S))$ of these cases has been already considered in \cite{EES} by Eggleton, Erd\H{o}s, and Skilton. It turns out that $\chi_B(S)$ and $\chi(G(\mathbb{Z}, S))$ are different in the majority of these cases. We summarize the comparison in the following table.

\begin{center}

\begingroup

\renewcommand{\arraystretch}{2} 
\begin{tabular}{|c | c | c|}
\hline
Finite set $S=$ & Chromatic number of  & Borel chromatic number \\
$\{\pm a_1, \dots, \pm a_n\}$ & distance graph $G(\mathbb{Z}, S)$ &  of $G(F(2^{\mathbb{Z}}), S)$ \\
\hline\hline 
general lower bound & 2 \cite{EES} & 3 (Theorem~\ref{thm:lower}) \\
\hline
general upper bound & $n+1$ \cite[12.2]{Wal} & $\displaystyle\left\lfloor\frac{3}{2}n\right\rfloor+1$ (Theorem~\ref{thm:upper}) \\
\hline
odd numbers & 2 \cite[Theorem 8]{EES} & 3 (Corollary~\ref{cor:odd}) \\
\hline
$\{\pm 1, \pm 2,\dots, \pm n\}$ & $n+1$ \cite[Theorem 9]{EES} & $n+2$ (Corollary~\ref{cor:1ton}) \\
\hline
 & 3 if $a_1+a_2$ is odd & 4 if $\{a_1, a_2\}=\{1,2\}$ \\
$S=\{\pm a_1, \pm a_2\}$& 2 if $a_1, a_2$ are odd & 3 otherwise \\
&\cite[Theorems 8 \& 10]{EES} &  (Theorem~\ref{thm:pair}) \\
\hline
\end{tabular}
\endgroup

\end{center}

\vskip 12pt
We are not able to give a general formula for $\chi_B(S)$ when $n\geq 3$, but our method can determine the value of $\chi_B(S)$ for many different types of $S$. In particular, we show that for any $n\geq 1$ and $3\leq k\leq n+2$, there is a generating set $S$ as above where $|S|=2n$ and $\chi_B(S)=k$. However, we do not have any example of an $S$ where $|S|=2n$ but $\chi_B(S)>n+2$.

We also consider generating sets of the following special form, which we call {\em punctured sets}, and obtain the following exact values.

\begin{theorem}\label{thm:1.3} Let $1\leq m<n$ be natural numbers, and let $S(n, m)=\{\pm 1,\dots, \pm n\}\setminus\{\pm m\}$. Then the following hold.
\begin{enumerate}
\item[\rm (1)] $\chi_B(S(n, 1))=\displaystyle\left\lfloor \frac{1}{2}(n+4)\right\rfloor$.
\item[\rm (2)] If $1<m<n<2m$, then $\chi_B(S(n,m))=m+1$.
\item[\rm (3)] If $n\geq 2m$, then 
$$ \chi_B(S(n,m))=\left\{\begin{array}{ll} \displaystyle\left\lceil\frac{n+m+1}{2}\right\rceil, & \mbox{if $n+m+1$ is odd and $\operatorname{gcd}(n+1,m)=1$;} \\ \\ \displaystyle\left\lceil\frac{n+m+3}{2}\right\rceil, & \mbox{otherwise.}
\end{array}\right.
$$
\end{enumerate}
\end{theorem}

The computation of $\chi(G(\mathbb{Z}, S))$ for punctured sets $S=S(n,m)$ was already considered in \cite{EES} by Eggleton, Erd\H{o}s, and Skilton, who proved some upper and lower bounds. Their exact values in all cases of $1\leq m<n$ were finally determined by Chang, Liu, and Zhu in \cite{CLZ}. Comparing their results with ours, we see that the values of $\chi(G(\mathbb{Z}, S(n,m)))$ and $\chi_B(S(n,m))$ are sometimes the same and sometimes different. For the purpose of the comparison, we summarize some answers in the following table; for the $n\geq 2m$ cases, we highlight only the contrast in the cases that $n$ is a multiple of $m$. 

\vskip 12pt 
\begin{center}

\begingroup

{\renewcommand{\arraystretch}{2} 
\begin{tabular}{|c | c | c|}
\hline
 Punctured set $S=$ & Chromatic number of  & Borel chromatic number \\
$\{\pm 1, \dots, \pm n\}\setminus\{\pm m\}$ & distance graph $G(\mathbb{Z}, S)$ &  of $G(F(2^{\mathbb{Z}}), S)$ \\
$1\leq m<n $ & & \\
\hline\hline 
$m=1$ & \ $\displaystyle\left\lfloor \frac{1}{2}(n+3)\right\rfloor$ & $\displaystyle\left\lfloor \frac{1}{2}(n+4)\right\rfloor$\\ & \cite{EES} & (Corollary~\ref{thm:new1})  \\
\hline
$1<m<n<2m$ & $m$ \cite[Theorem 4]{CLZ} & $m+1$ (Theorem~\ref{thm:new4}) \\
\hline
$m>1$, $\displaystyle\frac{n}{m}$ is even, & $\displaystyle\frac{n+m}{2}+1$  & $\displaystyle\frac{n+m}{2}+1$  \\

and $m$ is even& \cite[Theorem 8]{CLZ} & (Theorem~\ref{thm:new3}) \\
\hline
$m>1$, $\displaystyle\frac{n}{m}$ is even, & $\displaystyle\frac{n+m+1}{2}$  & $\displaystyle\frac{n+m+1}{2}+1$ \\
and $m$ is odd & \cite[Theorem 8]{CLZ} & (Theorem~\ref{thm:new3}) \\
\hline
$m>1$ and $\displaystyle\frac{n}{m}$ is odd& $\displaystyle\frac{n+m}{2}+1$  & $\displaystyle\frac{n+m}{2}+1$ \\
& \cite[Theorem 8]{CLZ} & (Theorem~\ref{thm:new3}) \\ 
\hline
\end{tabular}}
\endgroup
\end{center}

\vskip 12pt
It is also curious that $\chi(G(\mathbb{Z}, S))\leq \chi_B(S)\leq \chi(G(\mathbb{Z}, S))+1$ holds in all cases when we are able to determine the exact values. The first inequality is of course obvious. However, we do not have a proof of the second inequality.

The rest of the paper is organized as follows. In Section~\ref{sec:2} we give a different proof of the theorem of Greb\'{\i}k--Rozho\v{n} on the coincidence of the Borel and continuous chromatic numbers of $G(F(2^\mathbb{Z}), S)$. As a corollary of the proof, we develop a method of computation of $\chi_B(S)$ in terms of asymptotic graph theory. This method will be used throughout the rest of the paper.  Starting from Section~\ref{sec:5} we turn to the problem of determining the values of $\chi_B(S)$. In Section~\ref{sec:5} we give some upper and lower bounds, in Section~\ref{sec:6} we treat pairs, in Section ~\ref{sec:7} we treat triples, and finally in Section~\ref{sec:8} we consider punctured sets.

{\em Acknowledgments.} We thank Anton Bernshteyn, Jan Greb\'{\i}k, Steve Jackson, Shujin Qu,  Xu Wang and Ming Xiao for helpful discussions on the topic of this paper.

\section{Borel combinatorics and asymptotic graph theory}\label{sec:2}

In this section we give a different proof of the theorem of Greb\'{\i}k--Rozho\v{n} \cite{GR} about the coincidence of the Borel and the continuous combinatorics of $G(F(2^\mathbb{Z}), S)$ that can be formulated as LCL problems. By using some results of \cite{GJKS} we are able to further connect the Borel combinatorics with some very concrete statements about the existence of combinatorial objects on sufficiently large finite graphs of a special kind. This gives us a finitary method to compute the Borel chromatic numbers $\chi_B(S)$.

\subsection{Preliminaries in descriptive set theory, topological dynamics, and graph theory}
In this subsection we recall some preliminary concepts and results in descriptive set theory, topological dynamics, and graph theory. 

A topological space $X$ is {\em Polish} if it is separable and completely metrizable. If $X$ is a Polish space and $Y\subseteq X$, then the subspace topology on $Y$ is Polish if and only if $Y$ is a $G_\delta$ subset in $X$, i.e., $Y$ is the intersection of countably many open subsets of $X$. The collection of all {\em Borel sets} in a Polish space $X$ is the smallest $\sigma$-algebra generated by the collection of all open subsets of $X$. If $X, Y$ are Polish spaces, then $f\colon X\to Y$ is {\em continuous} if for any open subset $U\subseteq Y$, $f^{-1}(U)$ is an open subset of $X$, and $f$ is {\em Borel-measurable}, or simply {\em Borel}, if for any open subset $U\subseteq Y$, $f^{-1}(U)$ is a Borel subset of $X$.

Let $X$ be a Polish space and $A\subseteq X$. We say that $A$ is {\em comeager} if $A$ contains a dense $G_\delta$ subset of $X$. By definition, a comeager set is dense. By the Baire category theorem, in a Polish space the intersection of countably many comeager sets is still comeager. If $Y$ is another Polish space and $f\colon X\to Y$ is Borel, then there is a comeager $C\subseteq X$ such that $f\upharpoonright C\colon C\to Y$ is continuous.

We consider the shift space $\bsft^\Z=\{0,\dots,\bsft-1\}^\Z$ for some $\bsft\geq 1$. For $s\in \Z$ and $y\in \bsft$, the $\Z$-action is defined by the shift map
$$ (s\cdot y)(n)=y(n+s) $$
for any $n\in\Z$.  We equip $\bsft$ with the discrete topology, and $\bsft^\Z$ with the product topology. With this topology, $\bsft^\Z$ is a compact Polish space and the shift map is continuous. For any $y\in \bsft^Z$, the {\em orbit} of $y$ is 
$$ [y]=\left\{s\cdot y\colon s\in \Z\right\}. $$
A subset $Y$ of $\bsft^\Z$ is {\em shift-invariant} if for any $y\in Y$, we have $[y]\subseteq Y$. A closed shift-invariant subset of $\bsft^\Z$ is called a {\em subshift} of $\bsft^\Z$. 

In this context, a {\em pattern} $p$ is a nonempty finite word in the alphabet $\mathfrak{b}$. We regard a pattern $p$ as a function from $\{0, \dots, \ell(p)-1\}$ to $\{0,\dots, \bsft-1\}$, where $\ell(p)\geq 1$ is the {\em length} of $p$. If $y\in \bsft^\Z$ and $p$ is a pattern, then we say that $p$ {\em occurs} in $y$ if there is $n\in\Z$ such that for any $0\leq i<\ell(p)$, $p(i)=y(n+i)$; otherwise $p$ {\em does not occur} in $y$. 

Given patterns $p_1, \dots, p_k$, define a subshift 
$$ Y=Y_{(\bsft; p_1,\dots,p_k)}=\left\{y\in \bsft^\Z\colon \mbox{none of $p_1,\dots, p_k$ occur in $y$}\right\}. $$
$Y$ is a closed shift-invariant subset of $\bsft^\Z$, hence a subshift. Here $Y$ is called the {\em subshift of finite type} described by $(\bsft; p_1,\dots, p_k)$, and the patterns $p_1,\dots, p_k$ in the description of $Y$ are called the {\em forbidden patterns}. Any subshift of finite type can be equivalently described by forbidden patterns of the same length. Thus, in our discussions, we tacitly assume that the forbidden patterns have the same length.

The {\em free part} of $\bsft^\Z$ is
$$ F(\bsft^\Z)=\left\{y\in \bsft^\Z\colon \mbox{for any nonzero $s\in \Z$, $s\cdot y\neq y$}\right\}. $$
$F(\bsft^\Z)$ is a shift-invariant subset of $\bsft^\Z$. However, it is not closed, but a $G_\delta$ subset, hence it is a Polish space (and it is no longer compact). In this paper, we only consider $F(2^\Z)$, the free part of $2^\Z$.

A map $f\colon F(2^\Z)\to Y$, where $Y$ is a subshift of $\bsft^\Z$, is {\em equivariant} if for any $s\in \Z$ and $x\in F(2^\Z)$, we have $f(s\cdot x)=s\cdot f(x)$. 

Let $S$ be a finite generating set for $\Z$, i.e., $\langle S\rangle=\Z$. We tacitly assume that $0\not\in S$ and $S$ is {\em symmetric}, i.e., $S=-S$. The {\em Cayley graph} $G(\mathbb{Z}, S)$ is the graph $(\Z, E_S)$, where $\{u,v\}\in E_S$ exactly when $u-v\in S$. The {\em Schreier graph} $G(2^\mathbb{Z}, S)$ is the graph $(2^\Z, W_S)$, where $\{x,y\}\in W_S$ exactly when there is $s\in S$ such that $s\cdot x=y$. No matter what the generating set $S$ is, the connected components of $G(2^\mathbb{Z}, S)$ are exactly the orbits of $2^\Z$. Thus it makes sense to speak of the Schreier graph on a shift-invariant subset of $2^\Z$, such as $F(2^\Z)$. In the rest of the paper, for simplicity of notation, we denote the graph $G(F(2^\mathbb{Z}), S)$ also by $G_S$. Note that because of the freeness, each connected component of $F(2^\Z)$ is isomorphic to the Cayley graph $G(\mathbb{Z}, S)$. 

In general, if $K$ is a set and $G=(V,E)$ is a graph, a {\em proper $K$-coloring} of $G$ is a function $c\colon V\to K$ for some some set $K$ such that whenever $uv\in E$, we have $c(u)\neq c(v)$. The {\em chromatic number} of $G$ is the least cardinality of $K$ such that there exists a proper $K$-coloring of $G$. When $G$ is a topological graph, as in the case of $G_S$ defined above, we may equip the set $K$ with the discrete topology and consider proper $K$-coloring functions from $G$ to $K$ that are continuous or Borel, and then define the corresponding notions of {\em continuous chromatic number} or {\em Borel chromatic number} of $G$. For all of the graphs we consider in this paper, all of the different versions of the chromatic numbers are finite.

Now, LCL problems on $G_S$ can be reformulated in terms of equivariant maps from $F(2^\Z)$ to some suitably defined subshifts of finite type. Here we will not define LCL problems in general, but give only an example of how the existence of a proper $\bsft$-coloring of $G_S$ corresponds to an equivariant map from $F(2^\Z)$ to a subshift of $\bsft^\Z$ of finite type. Let $\ell=\max S+1$. We define a finite set $F$ of forbidden patterns of length $\ell$ in the alphabet $\bsft$  by
$$ p\in F\iff \exists 0\leq i<\ell\ \exists s\in S\ (0\leq i+s<\ell \mbox{ and } p(i)=p(i+s)). $$
Enumerate $F$ as $p_1,\dots, p_k$. Then it is straightforward to check by definition that the following are equivalent:
\begin{enumerate}
\item[(1)] There exists an equivariant map $f\colon F(2^\Z)\to Y_{(\bsft; p_1, \dots, p_k)}$;
\item[(2)] There exists a proper $\bsft$-coloring of $G_S$.
\end{enumerate}
Additionally, if we enhance both (1) and (2) by requiring the maps (or the coloring functions) to be continuous or Borel, the enhanced versions are still equivalent. 

Thus, in this paper, by the {\em continuous} (or {\em Borel}) {\em combinatorics} of $F(2^\Z)$, we mean the problem of the existence of a continuous (or Borel, respectively) equivariant map from $F(2^\Z)$ to an arbitrary subshift of finite type.

\subsection{Borel combinatorics versus continuous combinatorics}\label{sec:3}

A characterization of the continuous combinatorics of $F(2^\Z)$ has already been given in \cite[Section 2.2]{GJKS}, where the main theorem \cite[Theorem 2.2.2]{GJKS} has been dubbed the Two-Tiles Theorem (we will recall the details below). Here we expand this characterization to include Borel combinatorics, and thus give another proof of the theorem of Greb\'{\i}k and Rozho\v{n} \cite{GR}. 

To state our main result we need to first recall some definitions from \cite[Section 2.2]{GJKS}. 

We first define an infinite family of finite directed graphs 
$$\left\{\Gamma_{n,p,q}\colon 1\leq n\leq p, q\right\},$$ 
which we call {\em two-tiles graphs}. For any integers $a, b$ with $1\leq a\leq b$, let $T(a,b)$ be a simple directed path of length $a+b-1$. Note that $T(a, b)$ consists of $a+b$ many vertices and $a+b-1$ many edges. When enumerating the vertices of $T(a,b)$, we always start with the unique vertex with $0$ indegree and follow the directions of the edges. With this convention we may speak of the first $n$ vertices, the last $n$ vertices, etc. Now, to define $\Gamma_{n,p,q}$ we fix $1\leq n\leq p, q$. Let $T_1(n, p, q)$ be a copy of $T(n,p)$ and let $T_2(n,p,q)$ be a copy of $T(n,q)$. Then let $\Gamma_{n,p,q}$ be the directed graph obtained from first taking the disjoint union of $T_1(n,p,q)$ and $T_2(n,p,q)$ and then identifying the following four subgraphs of $T_1(n,p, q)$ and $T_2(n,p,q)$:
\begin{itemize}
\item the subgraph of $T_1(n,p,q)$ induced by its first $n$ vertices,
\item the subgraph of $T_1(n,p,q)$ induced by its last $n$ vertices,
\item the subgraph of $T_2(n,p,q)$ induced by its first $n$ vertices, and
\item the subgraph of $T_2(n,p,q)$ induced by its last $n$ vertices.
\end{itemize}
Note that each of these subgraphs is a simple directed path of length $n-1$, and there is a unique isomorphism between any two of these subgraphs. Hence there is a unique way to implement the identification step of the construction. Figure~\ref{fig:onedimgamma} illustrates the graphs $T_1(3,7,9)$, $T_2(3,7,9)$ and $\Gamma_{3,7,9}$.

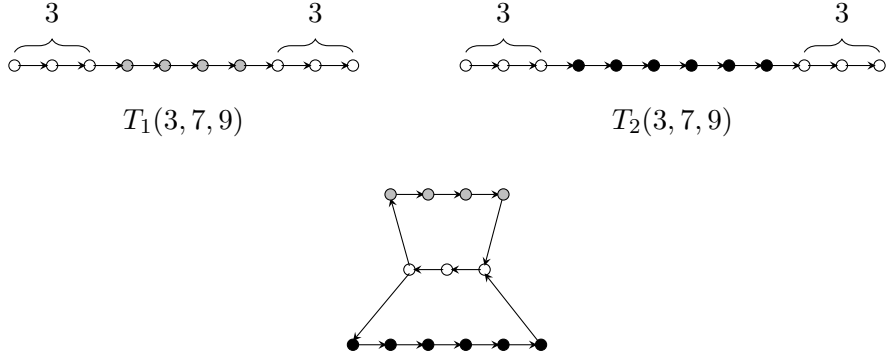
\begin{figure}[h]
\centering
\subfloat
{
\begin{tikzpicture}[scale=0.5]
\pgfmathsetmacro{\n}{3}
\pgfmathsetmacro{\p}{7}
\pgfmathsetmacro{\q}{9}
\pgfmathsetmacro{\a}{-7}
\pgfmathsetmacro{\b}{5}
\pgfmathsetmacro{\d}{4.5}
\pgfmathsetmacro{\e}{5.5}
\pgfmathsetmacro{\x}{\n+\p+\n-1}
\pgfmathsetmacro{\y}{\n+\q+\n-1}
\pgfmathsetmacro{\xa}{\n+\p-2}
\pgfmathsetmacro{\ya}{\n+\q-2}
\pgfmathsetmacro{\nm}{\n-1}
\pgfmathsetmacro{\pm}{\p-\n-1}
\pgfmathsetmacro{\qm}{\q-\n-1}
\foreach \i in {0,...,\nm}
{ \draw (\a+\i,0) circle (0.15);
}
\foreach \i in {0,...,\pm}
{ \draw[fill=lightgray] (\a+\n+\i,0) circle (0.15);
}
\foreach \i in {0,...,\nm}
{ \draw  (\a+\p+\i,0) circle (0.15);
}
\foreach \i in {0,...,\xa}
{ \draw[->,>=stealth] (\a+\i+0.1,0) -- (\a+\i+0.9,0);
}

\foreach \i in {0,...,\nm}
{ \draw  (\b+\i,0) circle (0.15);
}
\foreach \i in {0,...,\qm}
{ \draw[fill=black] (\b+\n+\i,0) circle (0.15);
}
\foreach \i in {0,...,\nm}
{ \draw  (\b+\q+\i,0) circle (0.15);
}
\foreach \i in {0,...,\ya}
{ \draw[->,>=stealth] (\b+\i+0.1,0) -- (\b+\i+0.9,0);
}

\draw[decorate, decoration={brace, mirror, amplitude=5pt},yshift=10pt] (\a+\n-1,0) -- (\a,0)
node [black,midway,yshift=15pt] {$3$};

\node at (\a+\d, -1.5) {$T_1(3,7,9)$};

\draw[decorate, decoration={brace, mirror, amplitude=5pt},yshift=10pt] (\a+\p+\n-1,0) -- (\a+\p,0)
node [black,midway,yshift=15pt] {$3$};

\draw[decorate, decoration={brace, mirror, amplitude=5pt},yshift=10pt] (\b+\n-1,0) -- (\b,0)
node [black,midway,yshift=15pt] {$3$};

\node at (\b+\e,-1.5) {$T_2(3,7,9)$};

\draw[decorate, decoration={brace, mirror, amplitude=5pt},yshift=10pt] (\b+\q+\n-1,0) -- (\b+\q,0)
node [black,midway,yshift=15pt] {$3$};

\end{tikzpicture}
}
\vspace{5pt}
\subfloat
{
\begin{tikzpicture}[scale=0.5]
\pgfmathsetmacro{\n}{3}
\pgfmathsetmacro{\p}{7}
\pgfmathsetmacro{\q}{9}
\pgfmathsetmacro{\a}{-10}
\pgfmathsetmacro{\b}{5}

\pgfmathsetmacro{\c}{-0.5}
\pgfmathsetmacro{\d}{2}
\pgfmathsetmacro{\e}{\p-\n-1}
\pgfmathsetmacro{\em}{\e-1}
\pgfmathsetmacro{\h}{0.1}

\pgfmathsetmacro{\f}{\q-\n-1}
\pgfmathsetmacro{\fm}{\q-\n-2}

\pgfmathsetmacro{\x}{\n+\p+\n-1}
\pgfmathsetmacro{\y}{\n+\q+\n-1}
\pgfmathsetmacro{\nm}{\n-1}
\pgfmathsetmacro{\nmm}{\n-2}
\pgfmathsetmacro{\pm}{\p-1}
\pgfmathsetmacro{\qm}{\q-1}

\foreach \i in {0,...,\em}
{ \draw[->,>=stealth] (\a+\c+\i+0.1,\d) -- (\a+\c+\i+0.9,\d);
}

\foreach \i in {0,...,\nmm}
{ \draw[->,>=stealth] (\a+\i+0.9,0) -- (\a+\i+0.1,0);
}

\foreach \i in {0,...,\fm}
{ \draw[->,>=stealth] (\a-1.5+\i+0.1,-\d) --  (\a-1.5+\i+0.9,-\d);
}

\draw[->,>=stealth]  (\a,\h) -- (\a+\c,\d-\h);
\draw[->,>=stealth] (\a+\c+\e,\d) -- (\a+\nm,\h);
\draw[->,>=stealth] (\a,-\h) -- (\a-1.5,-\d+\h);
\draw[->,>=stealth]  (\a-1.5+\f,-\d) -- (\a+\nm,-\h);

\foreach \i in {0,...,\e}
{ \draw[fill=lightgray] (\a+\c+\i,\d) circle (0.15);
}

\foreach \i in {0,...,\nm}
{ \draw (\a+\i,0) circle (0.15);
}

\foreach \i in {0,...,\f}
{ \draw[fill=black] (\a-1.5+\i,-\d) circle (0.15);
}

\end{tikzpicture}
}
\caption{\label{fig:onedimgamma}The construction of the two-tiles graph $\Gamma_{3,7,9}$.
}
\end{figure}

\begin{definition} Let $Y\subseteq \bsft^\Z$ be a subshift of finite type described by the sequence
$(\bsft;p_1,\dots,p_k)$ and let $g \colon \Gamma_{n,p,q}
\to \bsft$ be a map. We say that $g$  {\em respects} $Y$ if for any $1\leq i\leq k$ and for any directed path $v_1\cdots v_{\ell(p_i)}$ of length $\ell(p_i)-1$ in $\Gamma_{n,p,q}$, the word $g(v_1)\cdots g(v_{\ell(p_i)})$ of length $\ell(p_i)$ in alphabet $\bsft$ is not equal to $p_i$.
\end{definition}

The following is our main theorem of this section.

\begin{theorem} \label{thm:main}
Let $Y$ be the subshift of finite type described by $(\bsft; p_1,\dots, p_k)$. Then the following are equivalent.
\begin{enumerate}[label=\rm{(\arabic*)}, ref=\arabic*]
\item\label{odd}
There is a Borel equivariant map from $F(2^\Z)$ to $Y$.
\item\label{oda}
There is a continuous equivariant map from $F(2^\Z)$ to $Y$.
\item\label{odb}
There are integers $n,p,q$ with $\ell(p_1)\leq n<p,q$ and $\gcd(p,q)=1$, and there is $g \colon \Gamma_{n,p,q}\to \bsft$ which respects $Y$.
\item\label{odc}
For all $n\geq \ell(p_1)$ and for all sufficiently large $p,q>n$,
 there is $g \colon  \Gamma_{n,p,q}\to \bsft$ which respects $Y$.
\end{enumerate}
\end{theorem}

\begin{proof} The equivalence of (\ref{oda})--(\ref{odc}) is the Two-Tiles Theorem \cite[Theorem 2.2.2]{GJKS}. It is clear that (\ref{oda}) implies (\ref{odd}). It suffices to prove (\ref{odd}) implies (\ref{odc}).

Let $f\colon F(2^\Z)\to Y$ be a Borel equivariant map. There is a comeager set $C\subseteq F(2^\Z)$ so that $f\upharpoonright C\colon C\to Y$ is continuous. Let $D=\bigcap_{n\in\Z} n\cdot C$. Then $D$ is still comeager and $f\upharpoonright D\colon D\to Y$ is still continuous. Arbitrarily fix $x_0\in D$ and $n\geq \ell(p_1)$. The set
$$ U=\{ y\in \bsft^\Z\colon y(i)=f(x_0)(i) \mbox{ for all $0\leq i<n$}\} $$
is a basic open set in $Y$, and $x_0\in f^{-1}(U)\cap D$. By the continuity of $f\upharpoonright D$ on $D$, there is a basic open set $V$ of $F(2^\Z)$ such that $x_0\in V\cap D\subseteq f^{-1}(U)\cap D$. Without loss of generality, we may assume 
$$ V=\{z\in F(2^\Z)\colon z(i)=x_0(i) \mbox{ for all $a\leq i\leq b$}\} $$
for some $a<b\in \Z$. Now let $p, q>\max\{b-a, n\}$ be such that $\gcd(p,q)=1$. Then
$$ W=(-p\cdot V)\cap V\cap (q\cdot V) $$
is a nonempty open set in $F(2^\Z)$, and thus $W\cap D\neq\varnothing$ since $D$ is comeager. Now let $x\in W\cap D$ and $y=f(x)$. Then $x\in V\cap D$, and therefore $y\in U$. We also have $p\cdot x\in V\cap p\cdot D=V\cap D$ and therefore $$p\cdot y=p\cdot f(x)=f(p\cdot x)\in U. $$
Similarly, $-q\cdot x\in V\cap D$ and $-q\cdot y\in U$.

We now regard the Cayley graph of $\Z$ between $a-p$ and $b$ as $T_1(n,p,q)$ and between $a$ and $b+q$ as $T_2(n,p,q)$. Define $g\colon T_1(n, p,q)\to \bsft$ by $g(i)=y(i)$ for $a-p\leq i\leq b$, and define
$g\colon T_2(n,p,q)\to \bsft$ also by $g(i)=y(i)$ for $a\leq i\leq b+q$. It is easily seen that the values of $g$ for the first $n$ and last $n$ vertices of $T_1(n,p,q)$ and $T_2(n,p,q)$ coincide, and therefore this definition naturally induces a map $g\colon \Gamma_{n,p,q}\to \bsft$ which is well defined. Moreover, suppose $1\leq i\leq k$ and $v_1\dots v_{\ell(p_i)}$ is a directed path in $\Gamma_{n,p,q}$, then it is entirely contained in either $T_1(n,p,q)$ or $T_2(n,p,q)$. Thus $g(v_1)\dots g(v_{\ell(p_i)})$ is a word which occurs in $y$. Since $y\in Y$, we conclude that $g(v_1)\dots g(v_{\ell(p_i)})$ cannot be equal to $p_i$. This shows that $g$ respects $Y$. 
\end{proof}

Greb\'{\i}k and Rozho\v{n} \cite{GR} noted that the Borel combinatorics for $\mathbb{Z}$-actions is decidable. This also follows from Theorem~\ref{thm:main} and \cite[Theorem 4.2.1]{GJKS}.

\subsection{An algorithm for Borel chromatic numbers}\label{sec:4}

We describe an algorithm {\tt BPC} for the computation of $\chi_B(S)$ based on Theorem~\ref{thm:main}. This algorithm is not needed in any of the proofs in this paper, but is used for some examples of $S$ in obtaining the values of $\chi_B(S)$. We will present a generating set $S$ by the enumeration of its positive members in increasing order. We introduce the following notation.

\begin{definition} Denote by $\Sigma$ the set of all strictly increasing sequence of positive integers $a_1<\dots<a_n$ such that $\gcd(a_1,\dots, a_n)=1$. For $(a_1,\dots, a_n)\in \Sigma$, denote by $\chi(a_1,\dots, a_n)$ the chromatic number of $G(\mathbb{Z}, S)$ and by $\chi_B(a_1,\dots, a_n)$ the Borel chromatic number of the Schreier graph $G(F(2^\mathbb{Z}), S)$, where $S=\{\pm a_1,\dots, \pm a_n\}$. 
\end{definition}

We will show in the next section that for any $(a_1,\dots, a_n)\in \Sigma$, 
$$ 3\leq \chi_B(a_1,\dots, a_n)\leq 2n. $$
We can see directly that $\chi_B$ is a computable function. In fact, given $(a_1,\dots, a_n)\in\Sigma$, to compute $\chi_B(a_1,\dots, a_n)$ it suffices to successively check if there is a Borel proper $\bsft$-coloring for $\bsft=3, \dots, 2n$, until a positive answer is found; the smallest such $\bsft$ is $\chi_B(a_1,\dots, a_n)$.

In the following we describe an algorithm {\tt BPC} for checking whether there is a Borel proper coloring.

{\tt Input:} A sequence $(a_1,\dots, a_n)\in \Sigma$ and an integer $\bsft\in [3, 2n]$.

{\tt Step 1:} Let $m=a_n+1$ and let $\bsft^m$ be the set of all words of length $m$ in the alphabet $\bsft$. Form the set
$$ V=\left\{ u\in \bsft^m\colon u(j)\neq u(j+a_i) \mbox{ for all $1\leq i\leq n$ and $0\leq j\leq a_n-a_i$}\right\}. $$
Construct a directed graph $H=(V, E)$, where
$$ (u,v)\in E\iff \mbox{ for all $0\leq j<m$, $v(j)=u(j+1)$.} $$

{\tt Step 2:} For each connected component $C$ of $H$, compute the  value of the following function $f$:
$f(C)=\infty$, if there are no directed cycles in $C$; otherwise, $f(C)=\mbox{the $\gcd$ of the lengths of all directed cycles}$. 
 If for some connected component $C$ of $H$, $f(C)=1$, then return and output {\tt yes}. Otherwise, output {\tt no}.

This finishes the algorithm {\tt BPC}.

The correctness of the algorithm {\tt BPC} follows from the proof of \cite[Theorem 4.2.1]{GJKS}. The value in {\tt Step 2} is known as the {\em period} of the directed graph, or when the graph is represented by its adjacency matrix, the {\em period} of the matrix. This concept is well studied in algebraic graph theory, symbolic dynamics and matrix theory (see, e.g., \cite[Section 4.5]{LindBook} and \cite[Section 8.5]{HornBook}). There are well-known algorithms for {\tt Step 2} which runs in $O(N^3)$ time for directed graphs with $N$ vertices (for instance, this follows from Wielandt's theorem \cite[Corollary 8.5.8]{HornBook} on the primitivity of nonnegative matrices). However, the only upper bound we know for the size of a connected component of $H$ is $\bsft^m$. Thus {\tt BPC} runs in time which is exponential in $\max S$.

\section{Some bounds for Borel chromatic numbers}\label{sec:5}

In this section we prove some results that provide lower and upper bounds for the function $\chi_B(a_1,\dots, a_n)$. In some cases they give the correct values of $\chi_B$.

We will use the following notion in our proofs. We regard any finite simple directed path $T$ as an interval in $\Z$ with $xy\in T$ if and only if $x+1=y$. Thus we have an order on $T$ and can perform basic arithmetic operations on the vertices of $T$. Similarly for infinite or bi-infinite simple directed paths. Recall that for integers $a, b$ with $1\leq a\leq b$, $T(a, b)$ is the simple directed path of length $a+b-1$. We let $T(a, b)^{-}$ denote the subgraph of $T(a, b)$ induced by its first $a$ vertices, and $T(a, b)^+$ denote the subgraph of $T(a, b)$ induced by its last $a$ vertices.

\begin{definition} Let $S=\{\pm a_1,\dots, \pm a_n\}$ be a generating set of $\Z$ with $(a_1,\dots, a_n)\in \Sigma$. Let $G$ be a finite or infinite simple directed path. Let $\bsft\geq 1$ be an integer. An {\em $S$-coloration} of $G$ {\em with $\bsft$ colors} is a function $c\colon G\to \bsft$ such that for any $x, y\in G$, $c(x)\neq c(y)$ whenever $x-y\in S$. 
\end{definition}

The following lemma is an immediate corollary of Theorem~\ref{thm:main} and is our basic tool in the study of the function $\chi_B$. 

\begin{lemma}\label{lem:basic} For any generating set $S=\{\pm a_1,\dots, \pm a_n\}$ with $(a_1,\dots, a_n)\in \Sigma$, and for any integer $k\geq 1$, the following hold.
\begin{enumerate}
\item[\rm (a)] $\chi_B(a_1,\dots, a_n)>k$ if and only if for some $\ell\geq a_n+1$, there are arbitrarily large $p, q>\ell$ with $\gcd(p, q)=1$
such that there are no $S$-colorations $c_1$ of $T_1(\ell, p, q)$ and $c_2$ of $T_2(\ell, p, q)$ with $k$ colors such that the $c_1$ and $c_2$ agree on the subgraphs $T_1(\ell, p, q)^{\pm}$ and $T_2(\ell, p, q)^{\pm}$.
\item[\rm (b)] $\chi_B(a_1,\dots, a_n)\leq k$ if and only if for some $\ell\geq a_n+1$, there are $p, q>\ell$ with $\gcd(p, q)=1$ and $S$-colorations $c_1$ of $T_1(\ell, p, q)$ and $c_2$ of $T_2(\ell, p, q)$ with $k$ colors such that the $c_1$ and $c_2$ agree on the subgraphs $T_1(\ell, p, q)^{\pm}$ and $T_2(\ell, p, q)^{\pm}$.
\end{enumerate}
\end{lemma}

Now we consider lower bounds. 

\begin{definition} Given a generating set $S=\{\pm a_1,\dots, \pm a_n\}$ with $(a_1,\dots, a_n)\in \Sigma$, let $K_S$ denote the subgraph of the Cayley graph $G(\mathbb{Z}, S)$ induced by $\{0, a_1,\dots, a_n\}$, and call it the {\em core subgraph} of $G_S$. Denote by $\lambda_S$ the size of the largest clique in $K_S$, and by $\kappa_S$ the chromatic number of $K_S$.
\end{definition}

The following result gives some general lower bounds.

\begin{theorem}\label{thm:lower} For any generating set $S=\{\pm a_1,\dots, \pm a_n\}$ with $(a_1,\dots, a_n)\in \Sigma$, we have
$$\chi_B(a_1,\dots, a_n)\geq \kappa_S+1\geq \lambda_S+1\geq 3. $$
\end{theorem}

\begin{proof} It is clear that $\kappa_S\geq \lambda_S\geq 2$. It remains to show that $\chi_B(a_1,\dots, a_n)>\kappa_S$. Denote $\ell=a_n+1$ and $k=\kappa_S$. We will use Lemma~\ref{lem:basic} (a).

We first claim that for any $m\geq \ell$, an $S$-coloration $c$ of $T(\ell, m)$ with $k$ colors is determined by the restriction of $c$ on $T(\ell, m)^-$. This is to say, if $c, c'$ are both $S$-colorations of $T(\ell, m)$ with $k$ colors, and $c\upharpoonright T(\ell, m)^{-}=c'\upharpoonright T(\ell, m)^-$, then $c=c'$. For notational simplicity, identify $T(\ell, m)$ with the interval $[0, \ell+m)$ in $\Z$. To prove the claim, it suffices to show that $c(\ell)=c'(\ell)$; for the rest of the vertices in $T(\ell, m)$, the claim follows by an induction from left to right. Now consider the subgraph $X_S$ of the Cayley graph $G(\mathbb{Z}, S)$ induced by $\{\ell, \ell-a_1, \dots, \ell-a_n\}$. $X_S$ is isomorphic to $K_S$ via the isomorphism $x\mapsto \ell-x$. Thus the chromatic number of $X_S$ is $k$. Let $Y=X_S\setminus \{\ell\}$. Then $Y\subseteq T(\ell, m)^-$. We note that $c(Y)$ contains exactly $k-1$ colors. This is because, if $|c(Y)|<k-1$, then any extension of $c$ to a proper coloring on $X_S=Y\cup \{\ell\}$ uses $<k$ colors, contradicting the fact that the chromatic number of $X_S$ is $k$. It follows that $c(\ell)$ is the only color not used in $c(Y)$. Similarly, $c'(\ell)$ is the only color not used in $c'(Y)$. Since $c(Y)=c'(Y)$, we have $c(\ell)=c'(\ell)$ as required. 

Now we are ready to prove $\chi_B(a_1,\dots, a_n)>k$. Toward a contradiction, assume this fails. Then by Lemma~\ref{lem:basic} (a), for sufficiently large $p, q>\ell$ with $\gcd(p, q)=1$, there are $S$-colorations $c_1$ of $T_1(\ell, p, q)$ and $c_2$ of $T_2(\ell, p, q)$ with $k$ colors such that $c_1$ and $c_2$ agree on $T_1(\ell, p, q)^{\pm}$ and $T_2(\ell, p, q)^{\pm}$. Let $p$ and $q=p+a_1$ be sufficiently large with $\gcd(p, q)=1$. Let $c_1$ and $c_2$ be given with the above property. Then by the above claim, $c_1$ is determined by $c_1\upharpoonright T_1(\ell, p, q)^-$ and $c_2$ is determined by $c_2\upharpoonright T_2(\ell, p, q)^{-}$. Since 
$$c_1\upharpoonright T_1(\ell, p, q)^-=c_2\upharpoonright T_2(\ell, p, q)^{-}=c_2\upharpoonright T_2(\ell, p, q)^+,$$
 we have that 
$$ c_1(p+a_1)=c_2(p+a_1)=c_2(q)=c_2(0). $$
However, $c_1(p)=c_1(0)=c_2(0)$. Thus we have $c_1(p)=c_1(p+a_1)$, contradicting the assumption that $c_1$ is an $S$-coloration.
\end{proof}

Next we turn to upper bounds. For $S=\{\pm a_1, \dots, \pm a_n\}$, it is known that $\chi(G(\mathbb{Z}, S))\leq n+1$ \cite[12.2]{Wal}. 

\begin{theorem}\label{thm:upper} For any $n>1$ and any $(a_1,\dots, a_n)\in \Sigma$, we have
$$ \chi_B(a_1,\dots, a_n)\leq \left\lfloor \displaystyle\frac{3}{2}n\right\rfloor +1. $$
\end{theorem}

\begin{proof} Let $S=\{\pm a_1, \dots, \pm a_n\}$, $\ell=a_n+1$ and $k=\lfloor \frac{3}{2}n\rfloor +1$. By Lemma~\ref{lem:basic} (b), it suffices to find $p, q>\ell$ with $\gcd(p, q)=1$ and $S$-colorations $c_1$ of $T_1(\ell, p, q)$ and $c_2$ of $T_2(\ell, p, q)$ with $k$ colors such that $c_1$ and $c_2$ agree on $T_1(\ell, p, q)^{\pm}$ and $T_2(\ell, p, q)^{\pm}$.

We first claim that for any $m\geq \ell$, there is an $S$-coloration $c$ of $T(\ell, m)$ with $n+1$ colors. We identify $T(\ell, m)$ with the interval $[0, \ell+m)$ in $\Z$ and define such an $S$-coloration by induction from left to right using a greedy algorithm. Suppose $c$ has been defined for all $j<i$. To define $c(i)$ we consider the subgraph $X_S$ of $G(\mathbb{Z}, S)$ induced by $\{i, i-a_1,\dots, i-a_n\}\cap [0,i]$. The vertex $i$ has degree at most $n$ in $X_S$, and therefore we can assign $c(i)$ to be the first unused color in the $n+1$ many colors. This finishes the definition of $c(i)$. 

Let $A$ be a set of $n+1$ many colors and $B$ be a set of $k-n-1$ many colors so that $A\cap B=\varnothing$. Now fix a particular $S$-coloration $c$ of $T(\ell, \ell)$ with colors in $A$ and let $d$ be the restriction of $c$ on $T(\ell, \ell)^-$. 
Arbitrarily fix $p, q>3\ell^2$ with $\gcd(p, q)=1$. To complete our proof, we only need to define $S$-colorations $c_1$ of $T_1(\ell, p, q)$ and $c_2$ of $T_2(\ell, p, q)$ with colors in $A\cup B$ so that
$$ c_1\upharpoonright T_1(\ell, p, q)^{\pm}=c_2\upharpoonright T_2(\ell, p, q)^{\pm}. $$
For this, first let 
$$ c_1\upharpoonright T_1(\ell, p, q)^{\pm}=c_2\upharpoonright T_2(\ell, p, q)^{\pm}=d. $$
This will guarantee that $c_1$ and $c_2$ are as required when we finish our definition as long as $c_1$ and $c_2$ are $S$-colorations. Next, we extend this partial definition of $c_1$ to the entire $T_1(\ell, p, q)$, which is identified with the interval $[0, \ell+p)$ in $\Z$, by the following procedure. 

Step 1: Extend the partial definition of $c_1$ to $[0, p-\ell)\cup [p, \ell+p)$ by induction from left to right using a greedy algorithm similar to the one for the above claim. The result is a partial $S$-coloration with colors in $A$, since $p>2\ell$. We denote the resulting partial $S$-coloration by $d_{-1}$. 

Step 2: For $0\leq i<\ell$, define by induction a partial $S$-coloration $d_i$ with colors in $A\cup B$ so that $d_i$ is defined exactly on $[0, p-\ell+i] \cup [p, \ell+p)$, and $d_i$ and $d_{i-1}$ agree on 
$$\big[0, p-(2i+3)\ell\big)\cup \big[p, \ell+p\big). $$
 In the end, let $c_1=d_{\ell-1}$; then $c_1$ is an $S$-coloration with $k$ colors on the entire $[0, \ell+p)$, and $d_i$ and $d_{-1}$ agree on $[0,\ell)\cup [p, \ell+p)$ since $p>3\ell^2$.

Assume $d_{i-1}$ has been defined for some $0\leq i<\ell$. In the inductive step, we define $d_i$. Let $Y_S$ be the subgraph of $G_S$ induced by the set
$$Y_i=\{p-\ell+i-a_n, \dots, p-\ell+i-a_1, p-\ell+i, p-\ell+i+a_1, \dots, p-\ell+i+a_n\}. $$
Note that by our inductive hypothesis, $d_{i-1}(p-\ell+i)$ is not defined. We consider two cases.

Case 1: The partial coloration $d_{i-1}\upharpoonright Y_i$ uses $<k$ many colors. In this case expand $d_{i-1}$ to $d_i$ by defining $d_i(p-\ell+i)$ to be an unused color in $A\cup B$. This finishes the inductive definition in this case.

Case 2: The partial coloration $d_{i-1}\upharpoonright Y_i$ uses all $k$ colors in $A\cup B$. 

In this case we make the following claim.
\begin{quote}
Claim (C2): There is $j\in Y_i$ with $j<p-\ell+i$ so that for any $j'\in Y_i$ with $j'\neq j$, either $d_{i-1}(j')$ is undefined or $d_{i-1}(j')\neq d_{i-1}(j)$. 
\end{quote}
If Claim (C2) fails, then for all $j\in Y_i$ with $j<p-\ell+i$, there is some $j'\in Y_i$ with $j'\neq j$ and $d_{i-1}(j')=d_{i-1}(j)$. Since there are $n$ many elements in $\{ j\in Y_i\colon j<p-\ell+i\}$ and $d_{i-1}$ is defined for at most $2n$ many elements in $Y_i$,  we would use at most $\lfloor \frac{3}{2}n\rfloor=k-1$ many colors for $d_{i-1}\upharpoonright Y_i$, contradicting our case hypothesis.  

To continue our definition of $d_i$, we let $j\in Y_i$ be the least witness for Claim (C2), and define a partial $S$-coloration $d'$ by letting
$$ d'(x)=\left\{\begin{array}{ll} d_{i-1}(j), & \mbox{ if $x=p-\ell+i$,} \\ d_{i-1}(x), & \mbox{ if $x\in [0, j)\cup (j, p-\ell+i)\cup [p, \ell+p)$.}
\end{array}\right. $$
Intuitively, $d'$ is obtained from $d_{i-1}$ by giving the color of $j$ to $p-\ell+i$, so that $d'(p-\ell+i)$ is now defined but $d'(j)$ is now undefined. Because of the property of $j$ established by Claim (C2), $d'$ is a partial $S$-coloration.

Now redo the construction of the inductive step  with $p-\ell+i$ replaced by $j$ and $d_{i-1}$ replaced by $d'$. Namely, consider the subgraph of $G_S$ induced by $\{j-a_n, \dots, j-a_1, j, j+a_1,\dots, j+a_n\}$ and consider the two cases. If Case 1 occurs then the construction stops with a definition of the partial $S$-coloration for $j$. If Case 2 occurs then a new partial $S$-coloration is defined so that the color of $j$ is defined but a newly undefined vertex $<j$ is identified by an application of Claim (C2). 

Repeat this procedure until Case 1 finally occurs. We make the following claim.
\begin{quote}
Claim (C1): Case 1 happens before the newly undefined vertex falls into the interval $\big[0, p-(2i+3)\ell\big)$. 
\end{quote}
Granting Claim (C1), we know that this procedure terminates, and the resulting partial $S$-coloration $d_i$ satisfies the inductive hypothesis. 


To see that Claim (C1) holds, we note that by the inductive hypothesis, $d_{i-1}$ and $d_{-1}$ agree on $\big[0, p-(2i+1)\ell\big)\cup \big[p, \ell+p\big)$. It follows that on $\big[0, p-(2i+1)\ell\big)$, $d_{i-1}$ uses only colors in $A$.  Also note that each time Case 2 happens in the procedure, the undefined vertex is replaced by some other vertex that is $<\ell$ from the former vertex. Thus, if the Claim (C1) fails, there would be some partial $S$-coloration $d''$ resulting from the procedure and some $$j''\in [p-(2i+3)\ell, p-(2i+2)\ell)$$ such that $$d''\upharpoonright \{j''-a_n, \dots, j''-a_1, j'', j''+a_1, \dots, j''+a_n\}$$ uses all $k$ colors in $A\cup B$. However, note the following property guaranteed by our procedure:  for any $N<p-(2i+1)\ell$ and for any partial $S$-coloration $D$ defined in the procedure, if Case 1 has not occurred, then the set of colors used by $D\upharpoonright [0, N]$ is contained in the set of colors used by $d_{-1}\upharpoonright [0, N]$. Since $d_{i-1}\upharpoonright [0, p-(2i+1)\ell)$  uses only colors in $A$, it follows that $d''\upharpoonright \{j''-a_n, \dots, j''-a_1, j'', j''+a_1, \dots, j''+a_n\}$ uses only colors in $A$, and therefore cannot be all $k$ colors in $A\cup B$. This is a contradiction. Claim (C1) is thus proved.

This finishes the inductive definition of Step 2, and also the definition of $c_1$ on the entire $T_1(\ell, p, q)$.

A similar procedure gives a definition of an $S$-coloration $c_2$ on the entire $T_2(\ell, p, q)$. It is clear that $c_1$ and $c_2$ are as required.
\end{proof}

For notational simplicity in the rest of the proofs, we are going to assume that the set of colors is just $k=\{0, \dots, k-1\}$, and we will present $c_1$ and $c_2$ as words in the alphabet $k$. When $\ell, p, q$ are clear from the context, we are going to omit writing them and just write $T_1, T_2$, etc.

\begin{theorem} \label{thm:cong} Let $S=\{\pm a_1, \dots, \pm a_n\}$ be a generating set of $\Z$ with $(a_1, \dots, a_n)\in \Sigma$. Let $m\geq 1$ be an integer. If $S\cap (m+1)\Z=\varnothing$, then $\chi_B(a_1, \dots, a_n)\leq m+2$.
\end{theorem}

\begin{proof} Let $\ell\geq a_n+1$ be a multiple of $m+1$, say $\ell=M(m+1)$. Let $c_0$ be the word $(01\dots m)^M$. Let $b_0, b_1, \dots, b_m, N$ be sufficiently large integers so that
$$ a_n\ll b_0 \ll b_m \ll b_{m-1} \ll \cdots \ll b_2 \ll b_1 \ll N. $$
Also ensure $N> 2M+1$. Let $p=N(m+1)$ and $q=N(m+1)-1$. Then $p, q>\ell$ and $\gcd(p, q)=1$.

Now we construct $S$-colorations $c_1$ of $T_1$ and $c_2$ of $T_2$ as words in $m+2$. First let $c_1=(01\dots m)^N$. Since $S\cap (m+1)\Z=\varnothing$, $c_1$ is an $S$-coloration with $m+1$ colors. Also note that $c_1$ begins and ends with the word $c_0$. By Lemma~\ref{lem:basic} (b), all it remains is to define an $S$-coloration $c_2$ with $m+2$ colors which is of length $\ell+q$ so that $c_2$ begins and ends with $c_0$.

$c_2$ will be obtained from the following procedure with $m+2$ steps.

Step 0: Let $d_0$ be the word $c_0$ followed by the word $1\dots m$, and then followed by the word $(01\dots m)^{N-M-1}$. In other words, $d_0$ can be obtained from $c_1$ by removing the $0$ in its $|c_0|=M(m+1)$ coordinate. Let $u_0$ be the word following $c_0$ in $d_0$, that is, we write $d_0=c_0u_0$.

Step 1: $d_1$ is obtained from $d_0$ by replacing the first $b_1$ many occurrences of $1$ in $u_0$ by $m+1$. Write $d_1=c_0u_1$. 

Step $i$, $2\leq i\leq m$: Suppose the word $d_{i-1}$ was defined in the preceding step, where $d_{i-1}=c_0u_{i-1}$. Then $d_i$ is obtained from $d_{i-1}$ by replacing the first $b_i$ many occurrences of $i$ in $u_{i-1}$ by $i-1$.

Step $m+1$: Note that $d_m=c_0u_m$. $d_{m+1}$ is obtained by replacing the first $b_0$ many occurrences of $0$ in $u_m$ by $m$. We write $d_{m+1}=c_0u_{m+1}$.

Finally, let $c_2=d_{m+1}$. We verify that $c_2$ is an $S$-coloration with $m+2$ colors. First, by our construction, the color $m+1$ occurs in $c_2$ at the same positions as it occurs in $d_1$, and therefore they occur with period $m+1$. Since $S\cap (m+1)\Z=\varnothing$, there are no $x, y\in T_2$ with $c_2(x)=c_2(y)=m+1$ and $x-y\in S$. Next, note that $1$ does not occur in the first $b_1(m+1)$ many positions of $u_1$. Therefore, as long as $b_1-b_2>M$, the occurrences of $1$ in $d_2$ take place in two blocks which are more than $M(m+1)=\ell$ apart, and in each block it occurs periodically with period $m+1$. Also, $1$ occurs in $c_2$ at the same positions as it occurs in $d_2$. Thus, there are no $x, y\in T_2$ with $c_2(x)=c_2(y)=1$ and $x-y\in S$. The cases of the colors $2,\dots, m$ are similar. Finally, the color $0$ does not occur in the first $b_0(m+1)$ many positions of $u_{m+1}$. Thus, as long as $b_0>M$, the occurrences of $0$ in $d_{m+1}=c_2$ take place in two blocks that are more than $M(m+1)=\ell$ apart, and in each block it occurs periodically with period $m+1$. Thus there are no $x, y\in T_2$ with $c_2(x)=c_2(y)=0$ and $x-y\in S$. 

It is clear from our construction that $c_2$ begins with the word $c_0$. As long as $N-b_1>M$, we have that $c_2$ ends with $c_0$. Thus $c_2$ is as required. 
\end{proof}

\begin{corollary}\label{cor:odd} If $(a_1, \dots, a_n)\in \Sigma$ and all of $a_1, \dots, a_n$ are odd numbers, then $\chi_B(a_1,\dots, a_n)=3$.
\end{corollary}
\begin{proof} In the above theorem, set $m=1$; we get $\chi_B(a_1,\dots, a_n)\leq 3$. By Theorem~\ref{thm:lower}, $\chi_B(a_1,\dots, a_n)=3$.
\end{proof}

\begin{corollary}\label{cor:1ton} For any $n\geq 1$, $\chi_B(1, 2,\dots, n)=n+2$.
\end{corollary}

\begin{proof} Let $S=\{\pm 1, \pm 2, \dots, \pm n\}$. Then $\lambda_S\geq n+1$. Now the lower bound is given by Theorem~\ref{thm:lower}. The upper bound is by Theorem~\ref{thm:cong} since $\{1, \dots, n\}\cap (n+1)\Z=\varnothing$.
\end{proof}

\begin{corollary} For any integers $n\geq 1$ and $3\leq k\leq n+2$, there is $(a_1,\dots, a_n)\in \Sigma$ such that $\chi_B(a_1,\dots, a_n)=k$.
\end{corollary}

\begin{proof} Let $a_1, \dots, a_n$ enumerate the first $n$ postive integers in $\Z\setminus (k-1)\Z$. In particular, $a_i=i$ for $1\leq i\leq k-2$. Let $S=\{\pm a_1, \dots, \pm a_n\}$. Then $K_S$ contains a clique of size $k-1$. By Theorem~\ref{thm:lower}, $\chi_B(a_1,\dots, a_n)\geq k$. On the other hand, $S\cap (k-1)\Z=\varnothing$. By Theorem~\ref{thm:cong}, $\chi_B(a_1,\dots, a_n)\leq k$.
\end{proof}

The following is another special case. 

\begin{theorem}\label{thm:2a1} If $(a_1, \dots, a_n)\in \Sigma$ and $a_n<2a_1$, then $\chi_B(a_1,\dots, a_n)=3$.
\end{theorem}
\begin{proof} Let $S=\{\pm a_1, \dots, \pm a_n\}$. Let $\ell\geq a_n+1$ be a multiple of $3a_1$, say $\ell=3a_1M$. Let $N>2M+1$, $p=3a_1N$ and $q=3a_1N-1$. Then $p, q>\ell$ and $\gcd(p, q)=1$. Let $c_0=(0^{a_1}1^{a_1}2^{a_1})^M$. Let $c_1=(0^{a_1}1^{a_1}2^{a_1})^N$. Then $c_1$ starts and ends with $c_0$. To see that $c_1$ is an $S$-coloration with $3$ colors, suppose $x, y\in T_1$ with $x-y\in S$. Since $a_n<2a_1$, $x$ and $y$ are indices of a subword of the form $\alpha^{a_1}\beta^{a_1}\gamma^{a_1}$ for distinct colors $\alpha, \beta, \gamma$. Since $|x-y|\geq a_1$, we must have $c(x)\neq c(y)$. 

Let $c_2$ be obtained from $c_1$ with the $0$ in the $\ell$ coordinate removed. Since $N>2M+1$, $c_2$ still starts and ends with $c_0$. By a similar argument as above, $c_2$ is an $S$-coloration with $3$ colors.

Thus by Lemma~\ref{lem:basic} (b) we have shown $\chi_B(a_1, \dots, a_n)\leq 3$. By Theorem~\ref{thm:lower}, $\chi_B(a_1,\dots, a_n)=3$.
\end{proof}

\section{Borel chromatic numbers for pairs}\label{sec:6}
In this section we give a formula for Borel chromatic numbers in the special case of generating pairs.

\begin{theorem}\label{thm:pair} If $(a_1, a_2)\in \Sigma$, then
$$ \chi_B(a_1, a_2)=\left\{\begin{array}{ll}4, & \mbox{if $(a_1, a_2)=(1,2)$,} \\3, & \mbox{otherwise.}\end{array}\right. $$
\end{theorem}

\begin{proof} Let $S=\{\pm a_1, \pm a_2\}$. The fact $\chi_B(1,2)=4$ is by Corollary~\ref{cor:1ton}. We show that whenever $(a_1, a_2)\neq (1,2)$, we have $\chi_B(a_1, a_2)=3$. By Corollary~\ref{cor:odd}, we have $\chi_B(1, a_2)=3$ if $a_2$ is odd. Next we deal with the case $a_1=1$ and $a_2$ is even.

Case 1: $a_1=1$ and $a_2$ is even. 

Let $\ell=a_2+1$ and $t=a_2/2-1$. Let $c_0=(01)^t021$. Then $c_0$ has length $\ell$. Let $p=3\ell$ and $q=3\ell-2$. Then $p, q>\ell$ and $\gcd(p,q)=1$. Let $c_1=c_0^3$. By straightforward observation, $c_1$ is an $S$-coloration. To make this observation easier, we list the three copies of $c_0$ in three rows, where in each column, the adjacent entries correspond to positions which differ by $a_2=\ell-1$.
$$ \begin{array}{rccccccccccccl}
0 & 1 & 0 & 1 & \cdots & 0 & 1 & 0 & 1 & 0 & 2 & 1 &  & \\
   & 0 & 1 & 0 & \cdots &    & 0 & 1 & 0 & 1 & 0 & 2 & 1 & \\
   &    & 0 & 1 & \cdots &    &    & 0 & 1 & 0 & 1 & 0 & 2 & 1
   \end{array}
   $$
Let $c_2$ be obtained from $c_1$ by removing the subword $01$ in its $\ell+1$ and $\ell+2$ coordinates. Then $c_2$ is still an $S$-coloration by a straightforward observation of the following diagram.
$$ \begin{array}{rccccccccccccl}
0 & 1 & 0 & 1 & \cdots & 0 & 1 & 0 & 1 & 0 & 2 & 1 &  & \\
   & 0 & 1 & 0 & \cdots &    & 0 & 1 & 0 & 2 & 1 &  &  & \\
0   &  1  & 0 & 1 & \cdots &    &    & 0 & 1 & 0 & 2 & 1 &  &
   \end{array}
   $$

For the rest of the proof, we assume $a_1>1$. By Theorem~\ref{thm:2a1}, we may assume $a_2\geq 2a_1$. Since $(a_1, a_2)\in\Sigma$, $\gcd(a_1, a_2)=1$. Thus we can finish the proof by considering the following two cases. 

Case 2: For an odd number $m>1$, $ma_1<a_2<(m+1)a_1$.

In this part of the proof we redefine the values of $\ell, t, p, q, c_0, c_1$ and $c_2$. Let $b=a_2-ma_1$. Then $1\leq b<a_1$. Let $\ell=a_1+a_2=(m+1)a_1+b$ and $t=(m-1)/2\geq 1$. As before, we first define an $S$-coloration $c_0$ of length $\ell$ as follows. Let $u=0^{a_1}$, $v=1^{a_1-1}2$, $w=2^{a_1}$, and $s=1^b$. Then let
$$ c_0=(uv)^t usw. $$
It is easily seen that $u,v,w$ are of length $a_1$ and $c_0$ is then of length $2a_1(t+1)+b=\ell$. Also note that 
\begin{itemize}
\item $u$ and $v$ differ in every position,
\item $u$ and $w$ differ in every position, and
\item $sw$ and $us$ both have length $a_1+b$ and they differ in every position. 
\end{itemize}
These properties guarantee that $c_0$ is an $S$-coloration. Let $p=3\ell$ and $q=p-1$. Then $p, q>\ell$ and $\gcd(p, q)=1$. Let $c_1=c_0^3$. Then we may again list $c_1$ in three rows where the adjacent entries in a column correspond to positions which differ by $a_2=ma_1+b$, and observe that $c_1$ is an $S$-coloration. In the following diagram, we use square brackets to group $sw$ together in the first row and $us$ together in the second row, and use parentheses to group $sw$ together in the second row and $us$ together in the third row to show their correspondence in position.
$$  \begin{array}{rccccccccccccl}
u & v & u & v & \cdots & u & v & u & v & u & [s & w] &  & \\
   & u & v & u & \cdots &    & u & v & u & v & [u & (s] & w) & \\
   &    & u & v & \cdots &    &    & u & v & u & v & (u & s) & w
   \end{array}
   $$

Before we define $c_2$ we consider the word $c_0^*$, which is obtained from $c_0$ by removing its first entry (the $0$ in the $0$ coordinate). Then we may write $c_0^*$ as
$$ c_0^*=(u^*v^*)^t u^*s^*w^-, $$
where $u^*=0^{a_1-1}1$, $v^*=1^{a_1-2}20$, $s^*=1^{b-1}2$ and $w^-=2^{a_1-1}$. Now define $\tilde{v}=1^{a_1-2}2^2=1^{a_1-2}22$, and let
$$ \tilde{c}_0=(u^*\tilde{v})^tu^*sw^-. $$
Then note that 
\begin{itemize}
\item $u^*$ and $\tilde{v}$ are both of length $a_1$ and they differ in every position, 
\item $w^-$ is of length $a_1-1$ and $u^*$ and $w^-$ differ in every one of the first $a_1-1$ positions, 
\item $sw$ and $u^*s$ are both of length $a_1+b$ and they differ in every position, and
\item $sw^-$ has length $a_1+b-1$ and $sw^-$ and $u^*s$ differ in every one of the first $a_1+b-1$ positions.
\end{itemize}
It follows that $\tilde{c_0}$ is an $S$-coloration. Moreover, we now have that 
\begin{itemize}
\item $v$ and $u^*$ both have length $a_1$ and they differ in every position, and
\item $u$ and $\tilde{v}$ both have length $a_1$ and they differ in every position.  
\end{itemize}
It follows that $c_0\tilde{c_0}$ is an $S$-coloration by the following diagram.
$$  \begin{array}{rccccccccccccl}
u & v & u & v & \cdots & u & v & u & v & u & [s & w] &  & \\
   & u^* & \tilde{v} & u^* & \cdots &    & u^* & \tilde{v} & u^* & \tilde{v} & [u^* & s] & w^- & 
      \end{array}
   $$
Now we may rewrite $\tilde{c}_0$ as
$$ \tilde{c}_0=u^{-}v(\tilde{u}v)^{t-1}\tilde{u}s\tilde{w}, $$
where $u^{-}=0^{a_1-1}$, $\tilde{u}=20^{a_1-1}$ and $\tilde{w}=12^{a_1-1}$. Note that
\begin{itemize}
\item $\tilde{u}$ and $v$ both have length $a_1$ and they differ in every position, and
\item $s\tilde{w}$ and $us$ both have length $a_1+b$ and they differ in every position.
\end{itemize}
It follows that $\tilde{c}_0c_0$ is also an $S$-coloration by the following diagram.
$$  \begin{array}{rccccccccccccl}
    u^{-} & v& \tilde{u} & \cdots &    & \tilde{u} &v & \tilde{u} & v & \tilde{u} & [s & \tilde{w}] & \\  
& u & v & \cdots &  & & u & v & u & v & [u & s] & w &       \end{array}
   $$
Finally, define $c_2=c_0\tilde{c}_0c_0$. Then $c_2$ is an $S$-coloration of length $q$. This finishes the proof for Case 2.

Case 3: For an even number $m>1$, $ma_1<a_2<(m+1)a_1$. 

In this part of the proof we again redefine the values of $\ell, t, p, q, c_0, c_1$ and $c_2$. But we will keep the definitions of $b, u, v, s, u^*, v^*, s^*, \tilde{v}, \tilde{u}$ and $u^{-}$. For the convenience of the reader, we summarize their definitions in the following.
$$\begin{array}{llll}
u=0^{a_1} & u^*=0^{a_1-1}1 & \tilde{u}=20^{a_1-1} & u^-=0^{a_1-1}\\
v=1^{a_1-1}2 & v^*=1^{a_1-2}20 & \tilde{v}=1^{a_1-2}22 & \\
s=1^b & s^*=1^{b-1}2 & & 
\end{array}
$$
We let $b=a_2-ma_1$ and $d=a_1-b$. Let $t=m/2-1$ and $\ell=a_1+a_2$. Note that
$$\ell=2a_1t+a_1+b+d+b+a_1.$$
Let $p=3\ell$ and $q=3\ell-1$. Then $p, q>\ell$ and $\gcd(p,q)=1$. 

We add the following definitions.
$$ \begin{array}{llll}
x=2^d & x^*=2^{d-1}0 & & \\
y=0^b & y^*=0^{b-1}1 &  & \\
z=1^{a_1} & & \tilde{z}=01^{a_1-1} & z^-=1^{a_1-1}
\end{array}
$$ 
Then let
$$ c_0=(uv)^tusxyz $$
and $c_1=c_0^3$. Then $c_0$ has length $\ell$ and $c_1$ has length $3\ell$. Now note that
\begin{itemize}
\item $u$ and $z$ both have length $a_1$ and they differ in every position,
\item $s$ and $y$ both have length $b$  and they differ in every position,
\item $sx$ and $u$ both have length $a_1$ and they differ in every position, and
\item $xy$ and $z$ both have length $a_1$ and they differ in every position.
\end{itemize}
It follows that both $c_0$ and $c_1$ are $S$-colorations by the following diagram.
$$  \begin{array}{rcccccccccccccl}
    u & v & u & v & \cdots & u & v & u & v & u & sx &  y & z & \\
    & u & v & u & \cdots & & u& v& u & v& u& s & xy & z     \end{array}
   $$
Now let $c_0^*$ be obtained from $c_0$ by removing its first entry. Then we can rewrite $c_0^*$ as
$$ c_0^*=(u^*v^*)^tu^*s^*x^*y^*z^{-}.$$
Now let
$$ \tilde{c}_0=(u^*\tilde{v})^tu^*s^*x^*yz^{-}. $$
Note that
\begin{itemize}
\item $sx$ and $u^*$ still differ in every position, 
\item $s^*$ and $y$ still differ in every position, and
\item $x^*y$ and $z$ still differ in every position.
\end{itemize}
It follows that $c_0\tilde{c}_0$ is an $S$-coloration by the following diagram.
$$  \begin{array}{rcccccccccccccl}
    u & v & u & v & \cdots & u & v & u & v & u & sx &  y & z & \\
    & u^* & \tilde{v} & u^* & \cdots & & u^*& \tilde{v}& u^* & \tilde{v}& u^*& s^* & x^*y & z^{-}    \end{array}
   $$
$\tilde{c}_0$ can be rewritten as 
   $$ \tilde{c}_0=u^{-}v(\tilde{u}v)^{t-1}\tilde{u}sxy\tilde{z}. $$
   Note that $xy$ and $\tilde{z}$ still differ in every position. It follows that $\tilde{c}_0c_0$ is an $S$-coloration by the following diagram.
   $$  \begin{array}{rcccccccccccccl}
   u^{-} & v & \tilde{u} & v& \cdots & v & \tilde{u} & v &\tilde{u} & sx & y & \tilde{z} \\
   & u & v & u & \cdots &  & v & u & v & u & s &  xy & z & \\
    \end{array}
   $$
Now in all cases we have shown that $\chi_B(a_1, a_2)=3$ if $(a_1, a_2)\in \Sigma$ and $(a_1, a_2)\neq (1,2)$. The theorem is proved.
\end{proof}




\section{Borel chromatic numbers for triples}\label{sec:7}

We are not able to give a general formula for generating triples. Here we only treat a special case.

\begin{theorem}\label{thm:12b}
   If $b>2$, then 
   $$\chi_B(1,2,b)=\left\{\begin{array}{ll}5, & \mbox{if $b=3$,} \\4, & \mbox{if $b>3$.}\end{array}\right.$$
\end{theorem}
\begin{proof}
Let $S=\{1,2,b\}$. We have $\chi_B(1,2,3)=5$ by Corollary~\ref{cor:1ton}. 
Next we consider three cases of $b\pmod 3$ for $b>3$. 
    
    Case 1: $b\equiv 0\pmod 3$. Suppose $b=3k$ for $k>1$. Let $\ell=b+1$, $p=2\ell +1$ and $q=2\ell$. Then $\text{gcd}(p,q)=1$. Let $u=021$ and $v=0321$. Let $c_0=u^{k-1}v$, $c_1=c_03c_0^2$ and $c_2=c_0^3$. Then it follows from the following diagrams that both $c_1$ and $c_2$ are S-colorations.  

$$ c_1: \begin{array}{rcccccccccccccccccl}
0 & 2 & 1 & 0 & 2 & 1 & \cdots  &  & & 0 & 2 & 1 & 0 & 3 & 2 & 1 &  &  &  \\
   & 3 & 0 & 2 & 1 & 0 & 2 & 1 & \cdots  & & & 0 & 2 & 1 & 0 & 3 & 2 & 1 &  \\
& &          & 0 & 2 & 1 & 0 & 2 & 1 & \cdots & & & 0 & 2 & 1 & 0 & 3 & 2 & 1 
   \end{array}
   $$
$$ c_2: \begin{array}{rccccccccccccccccl}
0 & 2 & 1 & 0 & 2 & 1 & \cdots  & & 0 & 2 & 1 & 0 & 3 & 2 & 1 &  &  &  \\
    & 0 & 2 & 1 & 0 & 2 & 1 & \cdots & & 0 & 2 & 1 & 0 & 3 & 2 & 1 & & \\
 &       & 0 & 2 & 1 & 0 & 2 & 1 & \cdots & & 0 & 2 & 1 & 0 & 3 & 2 & 1 &  
   \end{array}
   $$
    
Case 2: $b\equiv 1\pmod 3$. In this case $3\mid{(b+2)}$. Let $\ell=b+2=3k$ for $k>1$. Let $p=2\ell$ and $q=2\ell-1$. Then $\text{gcd}(p,q)=1$. Let $u=021$ and $v=320$. Let $c_0=u^{k}$, $c_1=c_0^3$ and $c_2=c_0v^{k-1}31c_0$. Then $c_1$ and $c_2$ are $S$-colorations by the following diagrams. 
$$ c_1: \begin{array}{rcccccccccccccl}
0 & 2 & 1 & 0 & 2 & 1 & \cdots & & 0 & 2 & 1 &  &     & &   \\
    & & 0 & 2 & 1 & 0 & \cdots &  & 2 & 1 & 0 & 2 & 1  & &    \\
& & &           & 0 & 2 & 1 & \cdots & & 0 & 2 & 1 & 0 & 2 & 1
   \end{array}
   $$    
$$ c_2: \begin{array}{rccccccccccccl}
0 & 2 & 1 & 0 & 2 & 1 & \cdots & & 0 & 2 & 1 &   & &      \\
   & & 3 & 2 & 0 & 3 & \cdots &  & 2& 0 & 3 & 1 & & \\
& &       & 0 & 2 & 1 & \cdots & & 0 & 2 & 1 & 0 & 2 & 1 
   \end{array}
   $$

Case 3: $b\equiv 2\pmod 3$.  Then $3\mid{(b+1)}$. Let $\ell=b+1=3k$, $p=2\ell +1$ and $q=2\ell$. Then $\text{gcd}(p,q)=1$.  Let $u=021$. Let $c_0=u^{k}$, $c_1=c_03c_0^2$ and $c_2=c_0^3$. Then $c_1$ and $c_1$ are S-colorations by the following diagrams. 
    $$ c_1: \begin{array}{rccccccccccccl}
0 & 2 & 1 & 0 & 2 & 1 & \cdots  & & 0 & 2 & 1 &  & &      \\
   & 3 & 0 & 2 & 1 & 0 & 2 & 1 & \cdots & & 0 & 2 & 1 &  \\
   & &       & 0 & 2 & 1 & 0 & 2 & \cdots & &  & 0 & 2 & 1 
   \end{array}
   $$
   $$ c_2: \begin{array}{rcccccccccccl}
0 & 2 & 1 & 0 & 2 & 1 & \cdots  & & 0 & 2 & 1 &  &       \\
    & 0 & 2 & 1 & 0 & 2 & 1 & \cdots & & 0 & 2 & 1  &   \\
    & &    0 & 2 & 1 & 0 & 2 & \cdots & & & 0 & 2 & 1
   \end{array}
   $$
This shows that $\chi_B(1, 2, b)\leq 4$. By Corollary~\ref{cor:1ton}, $\chi_B(1,2)=4$. Thus $\chi_B(1,2, b)=4$.
\end{proof}

Our methods in the previous sections also yield the following partial results.

\begin{proposition}\label{prop:4} Let $S=\{\pm a_1,\dots, \pm a_n\}$ with $(a_1,\dots, a_n)\in \Sigma$. Then the following hold.
\begin{enumerate}
\item[(i)] If $\{a_1,\dots, a_n\}\subseteq 1+3\Z$ then $\chi_B(a_1,\dots, a_n)=3$.
\item[(ii)] If $\{a_1,\dots, a_n\}\subseteq 2+3\Z$ then $\chi_B(a_1,\dots, a_n)=3$.
\end{enumerate}
\end{proposition}

\begin{proof} For (i), let $c_0$ be a sufficiently long $S$-coloration with repetitions of $012$. Let $c_1$ be a sufficiently long $S$-coloration with repetitions of $012$, and thus $c_1$ starts and ends with $c_0$. Let $c_2$ be obtaind from $c_1$ by removing the first $0$ after the occurrence of $c_0$. Then they witness $\chi_B(S)=3$. (ii) can be proved with a slight modification.
\end{proof}



Also, with a proof similar to that of Theorem~\ref{thm:pair} but significantly more tedious, we can establish the following result, whose proof we omit here.

\begin{theorem}\label{thm:triple} Let $S=\{\pm a_1, \pm a_2, \pm a_3\}$ with $(a_1, a_2, a_3)\in \Sigma$. If $\kappa_S\in\{3, 4\}$ then $\chi_B(a_1, a_2, a_3)=\kappa_S+1$.
\end{theorem}

Theorem~\ref{thm:12b} is a special case of Theorem~\ref{thm:triple}.  We remark that Theorem~\ref{thm:triple} fails when $\kappa_S=2$. In fact, by running the algorithm {\tt BPC} one can verify that $\chi_B(1, 5,8)>3$. Since for $S=\{\pm 1, \pm 5, \pm 8\}$, $S\cap 3\Z=\varnothing$, it follows from Theorem~\ref{thm:cong} that $\chi_B(1, 5, 8)\leq 4$. Thus $\chi_B(1, 5,8)=4$. However, $\kappa_S=2$.




\section{Borel chromatic numbers for punctured sets\label{sec:8}}
In this section we consider sets of the form $\{1, \dots, n\}\setminus \{m\}$ for $1\leq m<n$. We call these sets {\em punctured}. The chromatic numbers of the distance graph on $\mathbb{Z}$ with these sets as distance sets were studied in \cite{EES}. Here we consider the generating sets $S(n, m)=\{\pm 1,\dots, \pm n\}\setminus \{\pm m\}$ and compute the Borel chromatic numbers $\chi_B(S(n,m))$. 

We first consider a generalization of $S(n,1)$.



    
    \begin{theorem}\label{thm:new1}
Let $m,n$ be positive integers. Then 
    $$\chi_B(n+1, n+2, \dots, n+m)=\displaystyle\left\lfloor \frac{n+m}{n+1}\right\rfloor+2.$$ 
    \end{theorem}
    \begin{proof} Let $S=\{\pm(n+1), \dots, \pm(n+m)\}$, $k=\lfloor (n+m)/(n+1)\rfloor$ and $r=(n+m)-k(n+1)$. 
    



Let 
$$\ell=(n+1)(k+1)+r=(n+m)+(n+1)$$ and 
   $$c_0=0^{n+1}1^{n+1}\cdots{(k-1)}^{n+1}{k}^{n+1}{(k+1)}^{r}.$$ Let $p=2\ell+1$ and $q=2\ell$. Then $\text{gcd}(p,q)=1$. Let $c_1=c_0(k+1)c_0^2$ and $c_1=c_0^3$. It is straightforward to check that  $c_0$, $c_0(k+1)c_0$ and $c_0^2$ are $S$-colorations. It then follows that $c_1$ and $c_2$ are $S$-colorations. We have thus shown $\chi_B(S)\leq k+2$. 

   Since there is a clique of size $k+1$ in $K_S$, we have $\chi_B(S)\ge k+2$ by Theorem~\ref{thm:lower}.
   Hence, $\chi_B(S)=k+2$.
    \end{proof}

\begin{corollary}\label{cor:new1} For any positive integer $n$, 
$$\chi_B(S(n,1))=\displaystyle\left\lfloor\frac{n}{2}\right\rfloor+2=\displaystyle\left\lfloor \frac{1}{2}(n+4)\right\rfloor.$$ 
\end{corollary}

   

\begin{theorem}\label{thm:new4}
Suppose $1<m<n<2m$. Then $\chi_B(S(n,m))=m+1$.
\end{theorem}

\begin{proof} By Theorem~\ref{thm:lower}, we have $\chi_B(S(n,m))\geq m+1$. Next we show $\chi_B(S(n,m))\leq m+1$. 

Let $\ell=2m$, $p=(m+2)m+1$ and $q=(m+2)m$. Then $\mbox{gcd}(p,q)=1$. Define $c=12\cdots m$ and $c_1=c^{m+4}$. Then $c_1$ is clearly an $S(n,m)$-coloration of length $q+\ell$. 

We use a new color $0$ in the construction of an $S(n,m)$-coloration $c_2$ of length $p+\ell$. To define $c_2$, we first let $c_{\{i\}}$
for $1\leq i\leq m$ be the sequence obtained from $c$ by the following procedure:
\begin{enumerate}
\item[(1)] remove the color $i$ from $c$, and
\item[(2)] add $0$ to the front of the sequence.
\end{enumerate}
Note that each $c_{\{i\}}$ still has length $m$. Now let
$$ c_2=c^2c_{\{1\}}c_{\{2\}}\cdots c_{\{m\}}0c^2. $$
It is straightforward to check that $c_2$ is an $S(n,m)$-coloration of length $p+\ell$.
\end{proof}

The following theorem treats the remaining cases and is the main theorem of this section. The reader can compare it to \cite[Theorem 8]{CLZ} and find that the values are sometimes the same as $\chi(G(\mathbb{Z}, S(n,m)))$ and sometimes different from it.

\begin{theorem}\label{thm:new3}
Suppose $n\geq 2m$. Let $P=n+m+1$ and $g=\operatorname{gcd}(P, m)$. Then
$$ \chi_B(S(n,m))=\left\{\begin{array}{ll} \displaystyle\left\lceil \frac{P}{2}\right\rceil, & \mbox{ if $P$ is odd and $g=1$;} \\ \\
\displaystyle\left\lceil \frac{P}{2}\right\rceil+1, & \mbox{ otherwise.}
\end{array}\right. $$
\end{theorem}

The rest of this section is devoted to a proof of Theorem~\ref{thm:new3}. The proof is divided into three parts.

\subsection{Proof of the lower bounds}

We first verify that the formula in the theorem gives the correct lower bounds. Suppose $T$ is an interval in $\mathbb{Z}$ and $c$ is an $S(n,m)$-coloration of $T$. Then observe that, if $u,v\in T$ with $c(u)=c(v)$, then either $|u-v|=m$ or $|u-v|\geq n+1$. Furthermore, if $u, v, w\in T$ with $u<v<w$ and $c(u)=c(v)=c(w)$, then it cannot happen that $w-v=v-u=m$, since then we would have $w-u=2m\leq n$, contradicting the preceding observation. In fact, in this case, we have $w-u\geq m+n+1=P$. To summarize, we have that, in any $P$ consecutive positions, one color can occur at most twice. This implies that
$$ \chi_B(S(n,m))\geq \chi(G(\mathbb{Z}, S(n,m)))\geq \displaystyle\left\lceil \frac{P}{2}\right\rceil. $$
This gives the correct lower bound when $P$ is odd and $g=1$. Otherwise we consider two cases.

Case 1: $P$ is even. Let $P=2K$. Toward a contradiction, assume $\chi_B(S(n,m))=K$. Let $T$ be any interval in $\mathbb{Z}$ with more than $P$ positions. Let $c$ be an $S(n,m)$-coloration of $T$ with $K$ colors. Then in any $P$ consecutive positions in $T$, each of the $K$ colors occur exactly twice. This implies that $c$ is periodic on $T$ with period $P$. In other words, $c$ on $T$ is determined by its restriction on the first $P$ many positions. This contradicts Lemma~\ref{lem:basic} (b).

Case 2: $P$ is odd and $g>1$. Let $P=2K-1$. We first claim that for any $S(n,m)$-coloration $c$ of $\mathbb{Z}$ with $K$ colors, for any $x\in \mathbb{Z}$, we have either $c(x)=c(x+m)$ or $c(x)=c(x-m)$.

Otherwise, assume without loss of generality that $c(0)\neq c(m)$ and $c(0)\neq c(-m)$. Note that in any $P$ consecutive positions in $\mathbb{Z}$, there is exactly one color occuring once while each of the other ($K-1$ many) colors occurs exactly twice. Consider the interval $[-n,n]$. The color $c(0)$ occurs only once in this interval. It follows that, for $-n\leq s\leq m$, we have $c(s+P)=c(s)$. Now if $-m\leq u<v\leq n$ and $c(u)=c(v)$, then either $v-u=m$ or $v-u>n$. In the latter case, we have $v\geq m+1$ and $-n\leq v-P\leq -m-1$; hence $c(v-P)=c(v)=c(u)$ and $u-(v-P)=m$. 

Now consider the interval $[-m,n]$ as a directed graph with  the edge relation given by $\pm m\pmod P$. Then the graph consists of $g$ many disjoint odd cycles, where each cycle has $P/g$ many vertices. Take any cycle not containing $0$; the relation of being assigned the same $c$ color gives a perfect matching of this odd cycle, contradicting the fact that an odd cycle does not have perfect matchings. This proves the claim. 

Now if $c$ is Borel proper $K$-coloring of $G(F(2^\mathbb{Z}), S(n,m))$, define
$$E=\{x\in F(2^\mathbb{Z})\colon c(x)=c(x+m)\}. $$
Then we have $E+m=E^c$ and $E+2m=E$. This contradicts the ergodicity of the transformation $+2m$ on $F(2^\mathbb{Z})$. Therefore, we have that $\chi_B(S(n,m))\geq K+1$.

\subsection{The upper bound $\lceil P/2\rceil+1$}
Let $K=\lceil P/2\rceil+1$. In this second part of the proof, we show that $\chi_B(S(n,m))\leq K$. By Lemma~\ref{lem:basic} (b), it suffices to find sufficiently large $\ell, p, q$ with $\operatorname{gcd}(p, q)=1$ and $S(n,m)$-colorations $c_1$ on $T_1(\ell, p, q)$ and $c_2$ on $T_2(\ell, p, q)$ with $K$ colors such that $c_1$ and $c_2$ agree on $T_1(\ell, p, q)^{\pm}$ and $T_2(\ell, p, q)^{\pm}$.

Our strategy of the proof is to find two $S(n,m)$-colorations $d_0$ on some interval $[0, 2km)$ and $d_1$ on $[0, 2km]$ for some sufficiently large $k$, both with $K$ colors, such that $d_0d_0$, $d_0d_1$ and $d_1d_0$ continue to be $S(n,m)$-colorations. Once such $d_0$ and $d_1$ are defined, we let $\ell=2km$, $p=2\ell$ and $q=2\ell+1$, and $c_1=d_0^3$ and $c_2=d_0d_1d_0$. Then $c_1$ and $c_2$ are as required.

Before defining $d_0$ and $d_1$, we consider an auxiliary coloring $\pi$ of an interval $T$ in $\mathbb{Z}$ by two colors, $\mathsf{R}$ (red) and $\mathsf{B}$ (blue). We refer to such a coloring $\pi$ as an {\em $\mathsf{RB}$-pattern} on $T$. We say an $\mathsf{RB}$-pattern $\pi$ on $T$ is {\em adequate} if the following hold:
\begin{enumerate}
\item[(i)] if $x, x-m\in T$ and $\pi(x)=\mathsf{B}$, then $\pi(x-m)=\mathsf{R}$;
\item[(ii)] in any $P$ consecutive positions in $T$, there are at most $K$ many positions colored by $\mathsf{R}$. 
\end{enumerate}
Suppose $\pi$ is an adequate $\mathsf{RB}$-pattern on $T$. Let
$$ T'=T\setminus \{x\in [0,m)\colon \pi(x)=\mathsf{B}\}. $$
We define an $S(n,m)$-coloration $c'$ of $T'$ with $K$ colors (in fact, with colors in $\{0, 1,\dots, K-1\}$) as follows. 
\begin{enumerate}
\item Enumerate the elements of $\{x\in T\colon \pi(x)=\mathsf{R}\}$ in increasing order as $x_0, x_1, \dots$;
\item Assign $c'(x_i)=i\pmod K$;
\item For each $x\in T'$ with $\pi(x)=\mathsf{B}$, assign $c'(x)=c'(x-m)$.
\end{enumerate}
By clause (i) of the definition of adequacy, $c'$ is well defined on $T'$. To see that $c'$ is an $S(n,m)$-coloration, let $u, v\in T'$ with $u<v$ and $c'(u)=c'(v)$. If $\pi(u)=\pi(v)=\mathsf{R}$, then by clause (ii) of the definition of adequacy, we have $v-u\geq P$. If $\pi(u)=\pi(v)=\mathsf{B}$, then by the definition of $c'$, $u-m, v-m\in T'$ and $\pi(u-m)=\pi(v-m)=\mathsf{R}$; thus again $v-u=(v-m)-(u-m)\geq P$. If $\pi(u)=\mathsf{B}$ and $\pi(v)=\mathsf{R}$, then $v-(u-m)\geq P$, and thus $v-u\geq n+1$. If $\pi(u)=\mathsf{R}$ and $\pi(v)=\mathsf{B}$, then either $v-u=m$ or else $v-m-u\geq P$, in the latter case $v-u> n+1$. This verifies that $c'$ is an $S(n,m)$-coloration of $T'$.

To further obtain an $S(n,m)$-coloration of $T$ with $K$ colors, we introduce the following notion. For $T=[0,t)$, $t>P$, and $\pi$ an $\mathsf{RB}$-pattern on $T$, we say that $\pi$ is {\em cyclical} if the following hold:
\begin{enumerate}
\item[(a)] the cardinality of $\{x\in T\colon \pi(x)=\mathsf{R}\}$ is a multiple of $K$;
\item[(b)] letting $\pi^2$ be the $\mathsf{RB}$-pattern on $[0, 2t)$ obtained from repeating $\pi$ twice, $\pi^2$ is adequate.
\end{enumerate}
Now suppose $\pi$ is a cyclical $\mathsf{RB}$-pattern on $T$. We define an $S(n,m)$-coloration $c$ of $T$ with $K$ colors as follows. For $x\in T'$ just let $c(x)=c'(x)$. For $x\in T\setminus T'$, namely $x\in [0, m)$ and $\pi(x)=\mathsf{B}$, note that clause (b) of the definition of cyclicality implies that $\pi(t-m+x)=\mathsf{R}$. We let $c(x)=c'(t-m+x)$.  Then $c$ is well defined on $T$. 
Let $d'$ be the $S(n,m)$-coloration of $[0, 2t)\setminus \{x\in [0,m)\colon \pi(x)=\mathsf{B}\}$ defined as in the preceding paragraph. Then we also have $c(x)=d'(x+t)$ for all $x\in [0, t)$. This implies that $c$ is an $S(n,m)$-coloration. For a cyclical $\mathsf{RB}$-pattern $\pi$, we denote this $S(n,m)$-coloration by $C_\pi$.

We are now ready to define $d_0$ and $d_1$ by defining appropriate cyclical $\mathsf{RB}$-patterns. We consider two cases.

\subsubsection{\bf Case 1: $m$ is even}

Let $\pi_0$ and $\pi_1$ be two $\mathsf{RB}$-patterns of length $m$ defined as
$$ \pi_0=\mathsf{BR}\dots \mathsf{BR}, \quad \pi_1=\mathsf{RB}\dots\mathsf{RB}. $$
Let $k$ be sufficiently large (the exact extent of the largeness of $k$ is determined by the construction below) so that, letting 
$$ \theta_0=(\pi_0\pi_1)^k, $$
the cardinality of $\{x\in [0, 2km)\colon \theta_0(x)=\mathsf{R}\}$ is a multiple of $K$. Note that $\theta_0$ is periodic with period $2m$. It is easy to see that both $\theta_0$ and $\theta_0^2$ are adequate $\mathsf{RB}$-patterns. Thus $\theta_0$ is indeed a cyclical $\mathsf{RB}$-pattern. Let $d_0=C_{\theta_0}$. Then $d_0$ is an $S(n,m)$-coloration so that $d_0d_0$ continues to be an $S(n,m)$-coloration.

Next we define $d_1$ so that the length of $d_1$ is one more than that of $d_0$. For this we define a new cyclical $\mathsf{RB}$-pattern $\theta_1$ of length $2km+1$ and let $d_1=C_{\theta_1}$. $\theta_1$ will be obtained from $\theta_0$ by performing a number of modifications. We first describe the kind of modifications that we will perform.

Suppose $\pi$ is an $\mathsf{RB}$-pattern of an interval $T$ in $\mathbb{Z}$. Suppose $x\in T$ and for all $i\geq 0$, we have
$$ \pi(x+im)=\left\{\begin{array}{ll} \mathsf{R}, & \mbox{ if $i\geq 0$ is even,} \\ \mathsf{B}, & \mbox{ if $i\geq 1$ is odd.}
\end{array}\right. $$
Then define $\Phi(\pi, x)$ to be an $\mathsf{RB}$-pattern on $T$ given by
$$ \Phi(\pi, x)(y)=\left\{\begin{array}{ll} \mathsf{R}, & \mbox{ if $y=x+(i+1)m$ for $i\geq 0$ even,} \\
\mathsf{B}, & \mbox{ if $y=x+(i+1)m$ for $i\geq 1$ odd,} \\
\pi(y), & \mbox{ otherwise.}
\end{array}\right.
$$
The effect of the modification $\Phi(\pi, x)$ is summarized in the following table.
$$\begin{array}{ccccccc}
 & x & x+m & x+2m & x+3m & x+4m & \cdots \\
 \pi & \mathsf{R} & \mathsf{B} & \mathsf{R} & \mathsf{B} & \mathsf{R} & \cdots \\
 \Phi(\pi, x) & \mathsf{R} & \mathsf{R} & \mathsf{B} & \mathsf{R} & \mathsf{B} & \cdots
 \end{array}
 $$
We make the following observations about $\Phi(\theta_0, x)$. 

\begin{lemma}\label{lem:mod1} Suppose $x\equiv m-1\pmod{2m}$. Then for any $y\in [0, 2km-(x+2m))$, we have
$$ \Phi(\theta_0, x)(x+2m+y)=\theta_0(y). $$
\end{lemma}
\begin{proof} The lemma is clear from a direct inspection of the following table:
$$ \begin{array}{ccccccc}
& \pi_0 & \pi_1 & \pi_0 & \pi_1 & \pi_0 & \cdots \\
\theta_0 & \mathsf{BR}\cdots\mathsf{BR} & \mathsf{RB}\cdots\mathsf{RB} &  \mathsf{BR}\cdots\mathsf{BR} & \mathsf{RB}\cdots\mathsf{RB} &  \mathsf{BR}\cdots\mathsf{BR} & \cdots \\
\Phi(\theta_0, x) & \mathsf{BR}\cdots\mathsf{B}\underline{\mathsf{R}} & \mathsf{RB}\cdots\mathsf{R}\underline{\mathsf{R}} &
\mathsf{BR}\cdots\mathsf{B}[\underline{\mathsf{B}} & \mathsf{RB}\cdots\mathsf{R}][\underline{\mathsf{R}} &\mathsf{BR}\cdots\mathsf{B}]\underline{\mathsf{B}} & \cdots 
\end{array}
$$
In the table the underlined positions are $x, x+m, x+2m, x+3m, x+4m, \dots$, which are the positions where $\Phi(\theta_0, x)$ made modifications. The brackets highlight the first $2m$ positions of an $\mathsf{RB}$-pattern identical to $\theta_0$.
\end{proof}

Lemma~\ref{lem:mod1} intuitively says that, if we perform the modification $\Phi$ to $\theta_0$ at a positition $x\equiv m-1\pmod{2m}$, then starting from position $x+2m$, the $\mathsf{RB}$-pattern $\theta_0$ reappears in the rest of $\Phi(\theta_0, x)$.

\begin{lemma} \label{lem:adq1} Suppose $x\equiv m-1\pmod{2m}$. Then $\Phi(\theta_0, x)$ is still adequate.
\end{lemma}

\begin{proof} Clause (i) of the definition of adequacy is clear from the definition of $\Phi(\theta_0, x)$. To verify clause (ii) of the definition of adequacy, consider $P$ consecutive positions $[y, y+P)$ for some $y\in [0, 2km)$. Let $N_{\mathsf{R}}(\theta_0)$ be the number of positions in $[y, y+P)$ which are colored by $\mathsf{R}$ in $\theta_0$, and let $N_{\mathsf{B}}(\theta_0)$ be similarly defined. For notational simplicity, let $\theta=\Phi(\theta_0, x)$. Also define $N_{\mathsf{R}}(\theta)$ and $N_{\mathsf{B}}(\theta)$ similarly. 

First suppose $P$ is even. Then by direct inspection, we have either $N_{\mathsf{R}}(\theta_0)\leq N_{\mathsf{B}}(\theta_0)$ or 
$N_{\mathsf{R}}(\theta_0)=N_{\mathsf{B}}(\theta_0)+2$. In the former case, we have $N_{\mathsf{R}}(\theta)\leq N_{\mathsf{B}}(\theta)+2$, which implies that $N_{\mathsf{R}}(\theta)\leq \lceil P/2\rceil+1=K$. In the latter case, the $\mathsf{RB}$-pattern on $[y, y+P)$ both starts and ends with $\mathsf{RR}$; if the colors are modified on this interval, then at least one of these positions of color $\mathsf{R}$ is modified to have color $\mathsf{B}$. Thus the number of positions of color $\mathsf{R}$ does not increase in the interval $[y, y+P)$. This implies again $N_{\mathsf{R}}(\theta)\leq N_{\mathsf{B}}(\theta)+2$ and $N_{\mathsf{R}}(\theta)\leq K$. 

Next suppose $P$ is odd. Then by direct inspection, we have either $N_{\mathsf{R}}(\theta_0)\leq N_{\mathsf{B}}(\theta_0)-1$ or 
$N_{\mathsf{R}}(\theta_0)= N_{\mathsf{B}}(\theta_0)+1$. In the former case, we have $N_{\mathsf{R}}(\theta)\leq N_{\mathsf{B}}(\theta)+1$, which implies that $N_{\mathsf{R}}(\theta)\leq K$. In the latter case, noting that $P\geq 3m$, and by direct inspection, we can see that the number of positions of color $\mathsf{R}$ does not increase in the interval $[y, y+P)$; thus again $N_{\mathsf{R}}(\theta)\leq K$.
\end{proof}

By similar arguments, we also have the following lemmas, which we state without proof.

\begin{lemma}\label{lem:mod2} Suppose $x\equiv m\pmod{2m}$. Then for any $y\in [0, 2km-1)$, we have
$$ \Phi(\theta_0, x)(x+1+y)=\theta_0(y). $$
\end{lemma}

\begin{lemma}\label{lem:adq2} Suppose $x\equiv m\pmod{2m}$. Then $\Phi(\theta_0, x)$ is still adequate.
\end{lemma}

Finally, we have the following observations.

\begin{lemma} \label{lem:mod3} Suppose $x\in [0, 2km)$ and $x\equiv m-1\pmod{2m}$. Let $k_0\geq 3$ be an odd number. Then for any $z\in [0,2km)$ with $z\geq x+k_0m+1$, we have
$$ \Phi(\Phi(\theta_0, x), x+k_0m)(z)=\theta_0(z). $$
\end{lemma}

\begin{proof} This follows from Lemmas~\ref{lem:mod1} and \ref{lem:mod2} and the periodicity of $\theta_0$. It can also be seen easily by direct inspection.
\end{proof}

\begin{lemma} \label{lem:adq3} Suppose $x\in [0, 2km)$ and $x\equiv m-1\pmod{2m}$. Then for a sufficiently large odd number $k_0$, $\Phi(\Phi(\theta_0, x), x+k_0m)$ is still adequate.
\end{lemma}

\begin{proof} This follows from Lemmas~\ref{lem:adq1} and \ref{lem:adq2}, noting that the second modification is performed on a part of $\Phi(\theta_0, x)$ which is a copy of $\theta_0$, where $x+k_0m\equiv m\pmod{2m}$ for this copy.
\end{proof}

We fix an odd number $k_0$ which makes Lemmas~\ref{lem:mod3} and \ref{lem:adq3} true. We are now ready to describe the modifications performed on $\theta_0$ to obtain $\theta_1$. Let $x_1, \dots, x_{m-1}$ and $y_1,\dots, y_L$, where $L=K-m/2$, be numbers in $[0, 2km)$ with the following properties:
\begin{enumerate}
\item $x_1\gg P$, and for each $1\leq i<m-1$, $x_{i+1}\gg x_i+P$;
\item $x_1\equiv m-1\pmod{2m}$, and for each $1\leq i\leq m-1$, $x_{i+1}-x_{i}\equiv m-1\pmod{2m}$;
\item $y_1\gg x_{m-1}$, and for each $1\leq j<L$, $y_{j+1}\gg y_j+P$;
\item $y_1-x_{m-1}\equiv m-1\pmod{2m}$, and for each $1\leq j<L$, $y_{j+1}-y_j\equiv m-1\pmod{2m}$;
\item $y_L\ll 2km-P$.
\end{enumerate}
Note that these conditions determine the largeness of $k$. 

Inductively define $\rho_0, \rho_1, \dots, \rho_{m-1}$ as follows:
$$ \rho_0=\theta_0, \quad \rho_{i}=\Phi(\rho_{i-1}, x_i) \mbox{ for $1\leq i\leq m-1$.} $$
Also inductively define $\eta_0, \eta_1, \dots, \eta_L$ by 
$$ \eta_0=\rho_{m-1}, \quad \eta_{j}=\Phi(\Phi(\eta_{j-1}, y_j), y_j+k_0m) \mbox{ for $1\leq j\leq L$.} $$
Eventually let $\theta_1=\eta_L\mathsf{B}$.

Informally, we perform a series of $\Phi$ modifications to $\theta_0$, and each modification is performed on a part of the $\mathsf{RB}$-pattern which is itself a copy of $\theta_0$. By Lemmas~\ref{lem:adq1} and \ref{lem:adq2} and the conditions for the sequences $x_1, \dots, x_{m-1}$ and $y_1, \dots, y_L$,  the resulting $\mathsf{RB}$-patterns $\rho_0=\theta_0, \rho_1, \dots, \rho_{m-1}=\eta_0, \eta_1,\dots, \eta_L$ are all adequate. 

By Lemma~\ref{lem:mod1}, for each $1\leq i\leq m-1$ and for any $y\in [0, 2km-(x_i+2m))$, 
$$ \rho_i(x_i+2m+y)=\Phi(\rho_{i-1}, x_i)(x_i+2m+y)=\rho_{i-1}(y). $$
Thus, we get that for any $z\in [0, 2km)$ with $z\geq x_{m-1}+2m$,  
$$\begin{array}{rcl}\rho_{m-1}(z)&=&\rho_{m-2}(z-x_{m-1}-2m) \\
&=&\rho_{m-3}(z-x_{m-2}-(m-1)-2Nm) \mbox{ for some $N$} \\
& & \mbox{ (since $x_{m-1}-x_{m-2}\equiv m-1\pmod{2m}$)} \\
&=&\cdots \\
&=&\rho_1(z-x_1-(m-2)(m-1)-2Nm) \mbox{ for some $N$} \\
&=&\rho_0(z-(m-1)^2-2Nm) \mbox{ for some $N$} \\
&=&\rho_0(z-1), 
\end{array}$$
where the last equality follows from the evenness of $m$ and the periodicity of $\rho_0=\theta_0$. It follows from these computations and Lemma~\ref{lem:mod3} that for all $z\in [0, 2km)$ with $z\geq y_L+k_0m+1$, 
$$ \eta_0(z)=\eta_1(z)=\cdots=\eta_L(z)=\rho_0(z-1)=\theta_0(z-1). $$
Thus for $z\in [0, 2km]$ with $z\geq y_L+k_0m+1$, $\theta_1(z)=\theta_0(z-1)$. This means that the last $P$ many positions of $\theta_1$ have the same $\mathsf{RB}$-pattern as the last $P$ many positions of $\theta_0$. Since it is also obivous that the first $P$ positions of $\theta_1$ have the same pattern as the first $P$ positions of $\theta_0$, we have that the $\mathsf{RB}$-patterns $\theta_0\theta_1$, $\theta_1\theta_0$ and $\theta_1\theta_1$ are all adequate. To verify that $\theta_1$ is cyclical, we only need to verify that the number of positions colored by $\mathsf{R}$ in $\theta_1$ is a multiple of $K$.

To do this, we count the number of positions colored by $\mathsf{R}$ in $\rho_i$ for $0\leq i\leq m-1$, which we denote by $N_i$. By direct inspection, we conclude that $N_i=N_{i-1}+1$ for odd $i\geq 1$ and $N_i=N_{i-1}$ for even $i\geq 2$. Thus $N_{m-1}-N_0=m/2$. Continue counting the number of positions colored by $\mathsf{R}$ in $\eta_j$ for $0\leq j\leq L$, which we denote by $M_j$. By inspection, we conclude that $M_j=M_{j-1}+1$ for all $1\leq j\leq L$. Thus 
$$ M_L-N_0=\displaystyle\frac{m}{2}+L=\displaystyle\frac{m}{2}+K-\displaystyle\frac{m}{2}=K. $$
Since $N_0$ is a multiple of $K$, so is $M_L$. Thus the number of positions colored by $\mathsf{R}$ in $\theta_1$ is a multiple of $K$.

We have shown that $\theta_1$ is cyclical. Thus $d_1=C_{\theta_1}$ is an $S(n,m)$-coloration with $K$ colors. By our construction, the coloring $d_0$ of the first $P$ positions coincides with $d_1$ of the first $P$ positions, and $d_0$ of the last $P$ positions coincides with $d_1$ of the last $P$ positions. Thus $d_0d_0$, $d_0d_1$ and $d_1d_0$ continue to be $S(n,m)$-colorations as required. (It is also true that $d_1d_1$ is an $S(n,m)$-coloration, but we do not need this fact for our construction.)

This finishes the proof of the upper bound in case $m$ is even.

\subsubsection{\bf Case 2: $m$ is odd} We use a similar construction as in the $m$ even case, but the details are different. 

For this case we use two kinds of modifications. The first kind is $\Phi(\pi, x)$ as before. We now define the second kind of modification. Consider an $\mathsf{RB}$-pattern $\pi$ on an interval $T$ in $\mathbb{Z}$. Suppose $x\in T$ such that $\pi(x)=\mathsf{B}$. Define $\Psi(\pi, x)$ by
$$ \Psi(\pi, x)(y)=\left\{\begin{array}{ll} \mathsf{R}, & \mbox{ if $y=x$,} \\ \pi(y), & \mbox{ otherwise.}\end{array}\right. $$
Since $\Psi(\pi, x)$ only modifies one position of the $\mathsf{RB}$-pattern $\pi$, it is clear that $\Psi(\pi, x)(y)=\pi(y)$ for all $y\geq x+1$. 

In the rest of this subsubsection we recycle the notations for the $\mathsf{RB}$-patterns we construct. Let
$$ \pi_0=\mathsf{RB\dots BR}, \quad \pi_1=\mathsf{BR\dots RB} $$
be both $\mathsf{RB}$-patterns of length $m$. Again let $k$ be sufficiently large so that, letting 
$$ \theta_0=(\pi_0\pi_1)^k, $$
the cardinality of $\{x\in [0, 2km)\colon \theta_0(x)=\mathsf{R}\}$ is a multiple of $K$. Again, $\theta_0$ has period $2m$, and is a cyclical $\mathsf{RB}$-pattern. Let $d_0=C_{\theta_0}$. Then $d_0$ is an $S(n,m)$-coloration so that $d_0d_0$ is still an $S(n,m)$-coloration.

In the rest of the proof we let $Q\geq P$ be the least even number, i.e., $Q=P$ if $P$ is even, and $Q=P+1$ is $P$ is odd. Then $K=Q/2+1$. 

For the record, we make the following easy observation about $\Psi(\theta_0, x)$. 

\begin{lemma}\label{lem:adq4} Suppose $x\in [0, 2km)$ and $\theta_0(x)=\mathsf{B}$. Then $\Psi(\theta_0, x)$ is still adequate.
\end{lemma}

\begin{proof} Clause (i) of the definition of adequacy is clear. For clause (ii) of the definition of adequacy, we verify the stronger statement that in any $Q$ consecutive positions in $\Psi(\theta_0, x)$, there are at most $K$ positions colored by $\mathsf{R}$. Consider positions in the interval $[y, y+Q)$. Since $Q$ is even, we have that if $x\not\in [y, y+Q)$, then the number of positions colored by $\mathsf{R}$ is $Q/2$, and if $x\in [y, y+Q)$, then the number of positions colored by $\mathsf{R}$ is $Q/2+1$. In both cases, there are at most $K$ positions in $[y, y+Q)$ colored by $\mathsf{R}$.
\end{proof}

Let $L=K-(m+1)/2$.  We will perform $L$ many $\Psi$ modifications later in the construction. That is, we will define $\eta_0, \eta_1, \dots, \eta_L$, where 
$$ \eta_{j}=\Psi(\eta_{j-1}, y_j) $$
for some suitably chosen $y_j$ for $1\leq j\leq L$, with $y_1\ll \dots\ll y_L\ll 2km$. Note that these modifications do not change the $\mathsf{RB}$-pattern $\eta_0$ beyond $y_L$, but increase the number of positions colored by $\mathsf{R}$ by $1$ with each application of $\Psi$. 

Similarly to the $m$ even case, we will also find numbers $x_0, \dots, x_{m-1}\in [0, 2km)$ so that $x_1\gg P$, $x_m\ll y_1$, and for each $1\leq i<m$, $x_{i+1}\gg x_i+P$. Then inductively define
$$ \rho_0=\theta_0, \quad \rho_i=\Phi(\rho_{i-1}, x_{i-1}) \mbox{ for $1\leq i\leq m$.} $$
Finally, let $\eta_0=\rho_m$ and $\theta_1=\eta_L\mathsf{B}$.

In order for this construction to work, we need to put more requirements on $x_0, \dots, x_{m-1}$. These requirements are dictated by the need to justify the adequacy of each $\rho_i$. Before describing these requirements, we perform an analysis on the number of positions colored by $R$ in an interval $[y, y+Q)$. 

For $0\leq i\leq m$ and $y\in [0, 2km-Q)$, let $N_i(y)$ be the number of positions in $[y, y+Q)$ colored by $\mathsf{R}$ in the $\mathsf{RB}$-pattern $\rho_i$; for $1\leq i\leq m$, let $\Delta_i(y)=N_i(y)-N_{i-1}(y)$. 

Define a function $\varepsilon\colon \mathbb{Z}\to \{0,1\}$ by
$$ \varepsilon(y)=\left\{\begin{array}{ll} 0, & \mbox{if $y\leq 0$ or if $y>0$ and $y\ (\mbox{mod}\ {2m})\in \{1,\dots, m\}$,} \\ 1, & \mbox{if $y>0$ and $y\ (\mbox{mod}\ {2m})\in \{m+1,\dots, 2m-1, 0\}$.}\end{array}\right. $$

Now fix $1\leq i\leq m$. Assume that for all $0\leq a<i-1$, $x_a\not\equiv x_{i-1}\ (\mbox{mod}\ m)$. For any $z\in (0, 2km]$, let $A_{i-1}(z)$ be the number of positions in $[0,z)$ colored by $\mathsf{R}$ in $\rho_{i-1}$, and let $A_i(z)$ be the number of positions in $[0,z)$ colored by $\mathsf{R}$ in $\rho_i=\Phi(\rho_{i-1}, x_{i-1})$. By the definition of $\Phi$, we observe that
$$ A_i(z)-A_{i-1}(z)=\varepsilon(z-x_{i-1}). $$
In fact, in this counting, because of the assumption on $x_{i-1}\ (\mbox{mod}\ m)$, we only need to focus on the positions congruent to $x_{i-1} \pmod{m}$. Thus, under the same assumption on $x_{i-1}\pmod{m}$, we have
$$\begin{array}{rcl} \Delta_i(y)&=&A_i(y+Q)-A_i(y)-(A_{i-1}(y+Q)-A_{i-1}(y)) \\
&=&A_i(y+Q)-A_{i-1}(y+Q)-(A_i(y)-A_{i-1}(y)) \\
&=&\varepsilon(y+Q-x_{i-1})-\varepsilon(y-x_{i-1}).
\end{array} $$

We note the following basic fact.

\begin{lemma}\label{lem:oddcase1} Assume that for all $0\leq a<i-1$, $x_a\not\equiv x_{i-1}\pmod{m}$. Suppose $Q\equiv 0\pmod{2m}$. Then $\Delta_i(y)\in\{0,1\}$ for all $y$. Moreover, $\Delta_i(y)=1$ only for some $y\in (x_{i-1}-Q+m, x_{i-1}]$. In particular, for all $y\in (x_i, 2km-Q)$, $\Delta_i(y)=0$.
\end{lemma}

\begin{proof} Note that $\epsilon(\cdot)$ is periodic on $\mathbb{N}$ with period $2m$. Thus for $y\geq x_{i-1}$, $\epsilon(y+Q-x_{i-1})=\epsilon(y-x_{i-1})$. 
\end{proof}

Now suppose $Q\equiv 0\pmod{2m}$. The requirements for $x_0, \dots, x_{m-1}$ are the following:
\begin{enumerate}
\item $x_1\gg P$, $x_m\ll 2km$, and for each $1\leq i\leq m$, $x_{i+1}\gg x_i+P$;
\item for each $0\leq i\leq m-1$, $x_i$ is even and $x_i\equiv i\pmod{m}$.
\end{enumerate}
Such $x_0, \dots, x_{m-1}$ are easy to find, since $m$ is odd, and thererfore for any $y\in \mathbb{N}$, one of $y$ and $y+m$ is even. We verify that each $\rho_i$, $1\leq i\leq m$, is adequate. First, clause (i) of the definition of adequacy is maintained by the $\Phi$ modification, so it holds for all $\rho_1, \dots, \rho_m$. For clause (ii) of the definition of adequacy, we show the stronger statement that in any $Q$ consecutive positions in $[0, 2km)$,  there are at most $K$ many positions colored by $\mathsf{R}$ in $\rho_i$. For this, we consider an interval $[y, y+Q)$ in $[0, 2km)$. Note that
$$ N_m(y)=N_0(y)+\Delta_1(y)+\cdots+\Delta_m(y). $$
We have $N_0(y)=Q/2$ from the definition of $\rho_0=\theta_0$. Since the intervals
$$ (x_0-Q+m, x_0], (x_1-Q+m, x_1], \dots, (x_{m-1}-Q+m, x_{m-1}] $$
are disjoint, we have by Lemma~\ref{lem:oddcase1} that for $y\leq x_{m-1}$, at most one of 
$$ \Delta_1(y), \dots, \Delta_m(y) $$
can be $1$, and the others are $0$. Thus $N_m(y)\leq Q/2+1=K$ for $y\leq x_{m-1}$. On the other hand, for $y>x_{m-1}$, Lemma~\ref{lem:oddcase1} gives that $\Delta_1(y)=\cdots=\Delta_m(y)=0$. Thus $N_m(y)\leq Q/2<K$ for $y>x_{m-1}$. This completes the proof that each $\rho_1, \dots, \rho_m$ is adequate. 

Our construction of $\rho_0, \rho_1, \dots, \rho_m$ makes it clear that, while $\rho_0$ contains entirely a repetition of the pattern $\pi_0\pi_1$, $\rho_m$ contains a sufficiently long repetition of the pattern $\pi_1\pi_0$ as suffix. This is because, each application of the $\Phi$ modification changes the parity of $\mathsf{R}$ and $\mathsf{B}$ on the tail end of a congruent $m$ class of indices. As per requirement (2) for $x_0, \dots, x_{m-1}$, $\Phi$ is performed exactly once for each congruent $m$ class. Therefore, the parity of $\mathsf{R}$ and $\mathsf{B}$ is the opposite for a tail end of $\rho_m$ compared to $\rho_0$. Formally, we have that for all $z\geq x_{m-1}+m$, $\rho_m(z)=\rho_0(z-1)$. 

Now we count the number of positions colored by $\mathsf{R}$ in $\rho_i$ for $0\leq i\leq m$, which is denoted by $N_i$. It is easy to see by inspection that, if $0< i<m$ is even, then $N_i=N_{i-1}$, and if $0<i\leq m$ is odd, then $N_i=N_{i-1}+1$. Thus $N_m=(m+1)/2$. For $0\leq j\leq L$, let $M_j$ be the number of positions colored by $\mathsf{R}$ in $\eta_j$. It is clear that for each $1\leq j\leq L$, $M_j=M_{j-1}+1$. Thus $M_L=N_m+L=N_0+K$. Since $N_0$ is a multiple of $K$, we conclude that $M_L$ is also a multiple of $K$.

Since each application of $\Psi$ only changes one position, we continue to have that for all $z>y_L$, 
$$ \eta_0(z)=\cdots=\eta_L(z)=\rho_0(z-1)=\theta_0(z-1). $$
This shows that $\theta_1=\eta_L\mathsf{B}$ is cyclical and that $\theta_1$ and $\theta_0$ agree on their first $Q$ positions and their last $Q$ positions. Let $d_1=C_{\theta_1}$. Then $d_0d_1$, $d_1d_0$ and $d_0d_0$ are all $S(n,m)$-colorations with at most $K$ colors. This finishes the construction and the proof for the case $Q\equiv 0\pmod{2m}$.

Next suppose $Q\not\equiv 0\pmod{2m}$. Let $\ell=m/g$, where $g=\operatorname{gcd}(Q, m)$. Then $1\leq \ell\leq m$ is the least positive integer such that $\ell Q\equiv 0\pmod{m}$. Since $Q$ is even and $m$ is odd, we in fact have that $\ell$ is the least positive integer such that $\ell Q\equiv 0\pmod{2m}$. Now the requirements for $x_0,\dots, x_{m-1}$ are the following:
\begin{enumerate}
\item[(1)] $x_1\gg P$, $x_m\ll 2km$, and for each $1\leq i\leq m$, $x_{i+1}\gg x_i+P$;
\item[(2')] for each $0\leq \gamma<g$, and $0\leq \lambda<\ell$, we have that $x_{\gamma\ell+\lambda}$ is even, 
$$x_{\gamma\ell}\equiv \gamma\pmod{g} \quad \mbox{ and } \quad x_{\gamma \ell+\lambda+1}\equiv x_{\gamma\ell+\lambda}+Q\pmod{2m}. $$
\end{enumerate}
Informally, we divide the indices $0, 1, \dots, m-1$ into $g$ groups, where for $0\leq \gamma<g$, the $\gamma$-th group consists of indices
$$ \gamma\ell, \gamma\ell+1, \dots, \gamma\ell+\ell-1. $$
The congruent $m$ classes of $x_0,\dots, x_{m-1}$ are still exactly $0, 1,\dots, m-1$ as in the above case of $Q\equiv 0\pmod{2m}$. Such $x_0,\dots, x_{m-1}$ can be found inductively by following (1) and (2'). 

For notational simplicity, consider the $0$-th group of indices, which correspond to the choice of 
$$ x_0, \dots, x_{\ell-1} $$
and $\mathsf{RB}$-patterns
$$ \rho_0, \rho_1, \dots, \rho_{\ell}. $$
Suppose $1\leq i\leq \ell$. Then since $Q\not\equiv 0\pmod{m}$, we have by (2') that for all $0\leq a<i-1$, $x_{i-1}-x_a\equiv (i-1-a)Q\not\equiv 0\pmod{m}$; thus $x_a\not\equiv x_{i-1}\pmod{m}$. Now let $[y, y+Q)$ be an interval in $[0, 2km)$. The observation before Lemma~\ref{lem:oddcase1} gives
$$\begin{array}{rcl} N_i(y)&=& N_0(y)+\Delta_1(y)+\cdots+\Delta_i(y) \\
&=& N_0(y)+\epsilon(y+Q-x_0)-\epsilon(y-x_0) \\
& &+\epsilon(y+Q-x_1)-\epsilon(y-x_1) \\
& &+\cdots \\
&&+\epsilon(y+Q-x_{i-1})-\epsilon(y-x_{i-1}). 
\end{array}
$$
By (2'), $x_{\lambda+1}\equiv x_{\lambda}+Q\pmod{2m}$ for each $0\leq \lambda<\ell$. Thus we have from Lemma~\ref{lem:oddcase1} that for $y>x_{\ell}$, $\epsilon(y+Q-x_{\lambda+1})=\epsilon(y-x_{\lambda})$ for each $0\leq \lambda<\ell$. Since $N_0(y)=Q/2$, it follows that for all $y>x_{\ell}$, 
$$ N_i(y)=N_0(y)+\epsilon(y+Q-x_0)-\epsilon(y-x_{i-1})\leq Q/2+1=K. $$
For $y\leq x_{\ell}$, a similar argument as in the case $Q\equiv 0\pmod{2m}$ gives also $N_i(y)\leq K$. This shows that $\rho_0, \rho_1, \dots, \rho_\ell$ are all adequate. Additionally, note that
$$ N_\ell(y)=N_0(y)+\epsilon(y+Q-x_0)-\epsilon(y-x_{\ell-1}). $$
By our construction, $x_{\ell-1}+Q\equiv x_0 \pmod{2m}$. Thus for $y>x_\ell+Q$, $N_\ell(y)=N_0(y)=Q/2$. It follows that similar computations apply to $N_i(y)$ for all $0\leq i\leq m$. Thus all of $\rho_0, \rho_1, \dots, \rho_m$ are adequate. 

Since the congruent $m$ classes of $x_0, \dots, x_{m-1}$ are exactly $0, 1, \dots, m-1$, similar to the above case of $Q\equiv 0\pmod{2m}$, we also have that for $z>x_{m-1}+m$, $\rho_m(z)=\rho_0(z-1)=\theta_0(z-1)$. Also similar to the above case, the number of positions colored by $\mathsf{R}$ in $\rho_m$ is $(m+1)/2$ more than the number of positions colored by $\mathsf{R}$ in $\rho_0$. The rest of the proof is similar to the above case of $Q\equiv 0\pmod{2m}$. 

We have thus proved the upper bound $\lceil P/2\rceil+1$ in the $m$ odd case.

\subsection{The upper bound $\lceil P/2\rceil$ when $P$ is odd and $g=1$} Assume throughout this subsection that $P$ is odd and $g=\operatorname{gcd}(P, m)=1$.  Let $Q=P+1$ and $K=\lceil P/2\rceil=Q/2$. Since $P$ is odd and $g=1$, we have $\operatorname{gcd}(P, 2m)=1$. Note that there is a unique $a\in [1, 2m)$ such that $aP\equiv 1\pmod{2m}$. In fact, there is $b\in \mathbb{N}$ such that $aP-2mb=1$. Thus it is easy to see that this unique $a$ is odd, and $\operatorname{gcd}(a,2m)=1$.

The construction follows the same outline as in the preceding subsection. Namely, we construct two $\mathsf{RB}$-patterns $\theta_0$ and $\theta_1$ whose length differ by $1$ and so that $\theta_0\theta_0$, $\theta_0\theta_1$, and $\theta_1\theta_0$ are all adequate.

First define $\theta_0$ on $[0, 4kKm)$, where $k$ is sufficiently large, by 
$$ \theta_0(x)=\mathsf{R}\iff ax\ (\mbox{mod}\ 2m)\in[0, m). $$
Namely, we also have
$$ \theta_0(x)=\mathsf{B}\iff ax\ (\mbox{mod}\ 2m)\in [m, 2m). $$
We verify that $\theta_0$ is cyclical. It is helpful to present $\theta_0$ in an $m\times 2kK$ matrix $U_0$ as follows. 
$$ U_0=\left(\begin{array}{cccc}
\theta_0(0) & \theta_0(m) & \theta_0(2m) & \cdots \\
\theta_0(1) & \theta_0(m+1) & \theta_0(2m+1) & \cdots \\
\vdots & \vdots & \vdots & \ddots \\
\theta_0(m-1) & \theta_0(2m-1) & \theta_0(3m-1) & \cdots
\end{array}\right). $$
For notational simplicity, we also define $\delta\colon \mathbb{Z}\to \{0,1\}$ by 
$$ \delta(x)=\left\{\begin{array}{ll} 1, & \mbox{ if $ax\ (\mbox{mod}\ 2m)\in [0, m)$,} \\ 0, & \mbox{ if $ax\ (\mbox{mod}\ 2m)\in[m,2m)$. } \end{array}\right. $$
Since $a$ is odd, we have that for all $x\in \mathbb{Z}$, $\delta(x+m)=1-\delta(x)$. It follows that $\theta_0$ has period $2m$. This also means that each row of $U_0$ has the alternating $\mathsf{RB}$-pattern $\dots\mathsf{RBRB}\dots$. This makes it clear that $\theta_0$ satisfies clause (i) of the definition of adequacy. It is clear that there are exactly $kK$ many positions colored by $\mathsf{R}$ in each row, and thus the number of positions colored by $\mathsf{R}$ in $\theta_0$ is $mkK$, a multiple of $K$. To see that $\theta_0$ and $\theta_0\theta_0$ are adequate, it suffices to verify clause (ii) of the definition of adequacy. For this, we perform some further analysis below.

First we make an observation.

\begin{lemma}\label{lem:oddcase2} For any $x\in \mathbb{Z}$, we have 
$$\begin{array}{rcl} \delta(x)-\delta(x-P)&=&\left\{\begin{array}{ll} 1, & \mbox{if $ax\equiv 0\pmod{2m}$, }\\ -1, & \mbox{if $ax\equiv m\pmod{2m}$, } \\ 0 & \mbox{otherwise.}\end{array}\right.  \\ \\
&=&\left\{\begin{array}{ll} 1, & \mbox{if $x\equiv 0\pmod{2m}$,} \\ -1, & \mbox{if $x\equiv m\pmod{2m}$,} \\ 0, & \mbox{otherwise.}\end{array}\right.
\end{array}$$
\end{lemma}

\begin{proof} Let $t=ax\pmod{2m}$. Then $a(x-P)\equiv t-1\pmod{2m}$ by our definition of $a$. The first equality can then be checked case by case. Since $\operatorname{gcd}(a, 2m)=1$, we get the second equality. 
\end{proof}

Also define $\alpha\colon \mathbb{Z}\to \{0,1\}$ by
$$ \alpha(x)=\left\{\begin{array}{ll} 1, & \mbox{if $x\ (\mbox{mod}\ 2m)\in \{0, m+1,m+2, \dots, 2m-1\}$,} \\ 0, & \mbox{otherwise.}
\end{array}\right.
$$
Then it is easy to check that for any $x\in \mathbb{Z}$, $\alpha(x+m)=1-\alpha(x)$ and 
$$ \alpha(x)-\alpha(x+1)=\left\{\begin{array}{ll} 1, & \mbox{if $x\equiv 0\pmod{2m}$,} \\ -1, & \mbox{if $x\equiv m\pmod{2m}$,} \\ 0, & \mbox{otherwise.}
\end{array}\right.
$$

Now consider a sub-interval $[y-P, y)$ in the interval $[0, 4kKm)$. Let $H(y)$ be the number of positions colored by $\mathsf{R}$ in $[y-P, y)$ in $\theta_0$. It is easy to see that for any $y\in [P, 4kKm-P-1)$, 
$$ H(y+1)=H(y)-\delta(y-P)+\delta(y). $$
Thus  by Lemma~\ref{lem:oddcase2} and the above observation about $\alpha$, we have
$$ H(y+1)-H(y)=\delta(y)-\delta(y-P)=\alpha(y)-\alpha(y+1). $$
This means that $H(y)+\alpha(y)$ is a constant for $y\in [P, 2kKm-P-1)$. Now note that
$$H(y)+H(y+m)=\displaystyle\sum_{x=y-P}^{y-1}\delta(x)+\sum_{x=y-P}^{y-1}\delta(x+m)=\sum_{x=y-P}^{y-1}(\delta(x)+\delta(x+m))=P. $$
Thus
$$\begin{array}{rcl} && H(y)+\alpha(y)+H(y+m)+\alpha(y+m) \\
&=&H(y)+H(y+m)+(\alpha(y)+\alpha(y+m)) \\
&=&P+1=Q. \end{array}$$
It follow that $H(y)+\alpha(y)=Q/2=K$. Thus $H(y)\leq K$. This shows that both $\theta_0$ and $\theta_0\theta_0$ are adequate, and $\theta_0$ is cyclical.

Next we define $\theta_1$ by performing the transformation $\Phi$ at appropriate positions. For this, we note some basic facts about $\Phi(\theta_0, x)$.

\begin{lemma}\label{lem:Phitheta} Suppose $x\equiv 0\pmod{2m}$. Then for all $y\in [x+m+P, 2kKm)$, we have
$$ \Phi(\theta_0, x)(y)=\theta_0(y-P). $$
\end{lemma}

\begin{proof} Note that $\theta_0(x)=\mathsf{R}$. Thus $\Phi(\theta_0, x)$ is well defined. Now suppose $y>x+m+P$. By Lemma~\ref{lem:oddcase2}, if $y\equiv 0\pmod{2m}$, then $\theta_0(y)=\mathsf{R}$ and $\theta_0(y-P)=\mathsf{B}$; if $y\equiv m\pmod{2m}$, then $\theta_0(y)=\mathsf{B}$ and $\theta_0(y-P)=\mathsf{R}$; otherwise $\theta_0(y)=\theta_0(y-P)$. In all cases we have exactly $\Phi(\theta_0,x)(y)=\theta_0(y-P)$.
\end{proof}

Intuitively, Lemma~\ref{lem:Phitheta} means that a tail end of $\Phi(\theta_0, x)$ coincides with a shift of $\theta_0$ by $P$.

\begin{lemma}\label{lem:adq5} Suppose $x\equiv 0\pmod{2m}$. Then $\Phi(\theta_0, x)$ is still adequate.
\end{lemma}

\begin{proof} Consider a subinterval $[y-P,y)$ of the interval $[0, 2kKm)$. Let $H(y)$ be the number of positions colored by $\mathsf{R}$ in $[y-P, y)$ in $\theta_0$, and let $H^*(y)$ be the number of positions colored by $\mathsf{R}$ in $[y-P, y)$ in $\Phi(\theta_0, x)$. 

If $y\leq x+m$, then $H^*(y)=H(y)\leq K$. By Lemma~\ref{lem:Phitheta}, if $y>x+m+P$, then $H^*(y)=H(y-P)\leq K$. Otherwise, $x+m\in [y-P, y)$. We consider two cases.

Case 1: $y-1-(x+m)\ (\mbox{mod}\ 2m)\in [0,m)$. In this case the $\mathsf{RB}$-pattern on the first row of $U_0$ between position $x$ and position $y-1$ is of the form $\mathsf{RBRB}\dots \mathsf{RBR}$. The corresponding $\mathsf{RB}$-pattern in $\Phi(\theta_0, x)$ is $\mathsf{RRBR}\dots\mathsf{BRB}$. This implies that $H^*(y)=H(y)\leq K$.

Case 2: $y-1-(x+m)\ (\mbox{mod}\ 2m)\in [m, 2m)$. In this case the $\mathsf{RB}$-pattern on the first row of $U_0$ between position $x$ and position $y-1$ is of the form $\mathsf{RBRB}\dots\mathsf{BRB}$. The corresponding $\mathsf{RB}$-pattern in $\Phi(\theta_0, x)$ is $\mathsf{RRBR}\dots\mathsf{RBR}$. Thus $H^*(y)=H(y)+1$. However, now we have $y\ (\mbox{mod}\ 2m)\in \{m+1, m+2, \dots, 2m-1,0\}$, and thus $\alpha(y)=1$. Since $H(y)+\alpha(y)=K$, we have $H^*(y)=H(y)+1=K$.
\end{proof}

Let $L=aQ-1$. We will perform $L$ many applications of the transformation $\Phi$ at positions $x_1, \dots, x_L$, where $x_1, \dots, x_L$ satisfy the following requirements:
\begin{enumerate}
\item $x_1\gg P$, $x_L\ll 2kKm$, and for each $1\leq i\leq L$, $x_{i+1}\gg x_i+P$;
\item for each $1\leq i\leq L$, $x_i+(i-1)P\equiv 0\pmod{2m}$.
\end{enumerate}
Inductively define
$$\rho_0=\theta_0, \quad \rho_i=\Phi(\rho_{i-1}, x_i) \mbox{ for all $1\leq i\leq L$.} $$
Finally let $\theta_1=\rho_L\mathsf{B}$. 

Note that for all $1<i\leq L$, $x_i-P\equiv x_{i-1}\pmod{2m}$. Then Lemma~\ref{lem:Phitheta} can be inductively applied to $\rho_{i-1}$ and $x_i$. For $y>x_L+m+P$, by Lemma~\ref{lem:Phitheta} we have
$$ \rho_L(y)=\rho_{L-1}(y-P)=\cdots=\rho_1(y-(L-1)P)=\rho_0(y-LP)=\theta_0(y-LP). $$
Since 
$$ a(y-LP)\equiv ay-aLP \equiv ay-L\equiv ay-(aP+a-1)\equiv a(y-1) \pmod{2m}, $$
we have 
$$ \theta_0(y-LP)=\theta_0(y-1). $$
This means that a tail end of $\theta_1$ coincides with a tail end of $\theta_0$.

It follows that $\theta_1\theta_1$ is still adequate. Now consider the $\mathsf{RB}$-pattern $\theta_1\theta_1$. Let $N$ be the number of positions in $\theta_1$ which are colored by $\mathsf{B}$. Let $S$ be the set of positions $x$ in $\theta_1$ such that both positions $x$ and $x+m$ are colored by $\mathsf{R}$. Then we know that $|S|=L$. Let $p=4kKm$ and $q=p+1$. Then we have $2q=2(2N+L)$. Thus the number of positions colored by $\mathsf{R}$ in $\theta_1$ is
$$ N+L=\displaystyle\frac{q-L}{2}+L=\frac{q+L}{2}=\frac{p}{2}+\frac{L+1}{2}=\frac{p}{2}+a\frac{Q}{2}=\frac{p}{2}+aK, $$
which is a multiple of $K$.

This shows that $\theta_1$ is cyclical, and our proof is complete.

\bigskip
{\bf AI Statement.} No generative AI was used in the research and the preparation of this paper.

\end{document}